\documentclass[11pt]{article}
\usepackage{amsfonts}
\usepackage{amssymb}
\usepackage{geometry}
\usepackage{setspace}
\usepackage{amsmath}
\usepackage{graphicx}
\usepackage{float}
\usepackage[truetex]{lscape}

\newtheorem{theorem}{Theorem}

\newtheorem{definition}{Definition}

\newtheorem{lemma}{Lemma}
\newtheorem{notation}{Notation}

\newtheorem{proposition}{Proposition}
\newtheorem{remark}{Remark}

\newenvironment{proof}[1][Proof]{\noindent\textbf{#1.} }{\ \rule{0.5em}{0.5em}}
\input{tcilatex}
\begin{document}

\begin{center}
\bigskip {\LARGE A Ladder Tournament}

\bigskip

Roland Pongou, Bertrand Tchantcho and Narcisse Tedjeugang\footnote{%
Pongou is at the Department of Economics of the University of Ottawa and a
visiting scholar in the Department of Economics of the Massachusetts
Institute of Technology; Tchantcho is at \'{E}cole Normale Sup\'{e}rieure
(UY1) and the Department of Economics of the Universit\'{e} de Cergy
Pontoise; and Tedjeugang is at \'{E}cole Normale Sup\'{e}rieure (UY1).
\par
Correspondence can be addressed to Pongou at: rpongou@uottawa.ca or
rpongou@mit.edu}

\bigskip

June 2015\bigskip
\end{center}

\textbf{Abstract}: Ladder tournaments are widely used to rank individuals in
real-world organizations and games. Their mathematical properties however
are still poorly understood. We formalize the ranking rule generated by a
ladder tournament, and we show that it is neither complete nor transitive in
general. If it is complete, then it is transitive and its asymmetric
component is a finite union of transitive tournaments. We also study the
relationship between an individual's rank and his performance as measured by
the frequency at which he is pivotal. We show an individual's pivotability
is a weakly increasing function of his rank.\pagebreak 

\section{Introduction}

\bigskip

This paper studies the mathematical properties of a \textit{game ladder}
involving a finite number of vertices or players. Players are listed as if
on the rungs of a ladder, and they compete against each other for
higher-ranked positions. Each player challenges each of the players above
him on the ladder, and if the lower-placed player wins, then the two players
swap positions on the ladder. The game goes on until no lower-placed player
wins a challenge against a higher-placed player. This tournament model is
used in real-life organizations to make decisions on hiring, firing, and
promotion.

We model a game ladder as a hierarchical organization $(N,T,f)$ where $N$ is
a non-empty finite set of vertices or players, $T=\{1,...,j\}$ an ordered
finite set of position types with cardinility $|T|=j\geq 2$, and $%
f:T^{N}\longrightarrow 
\mathbb{R}
$ a real-valued production function measuring organizational performance.
Positions are ordered in degree of importance, where type $1$ positions are
less important than type 2 positions, type $2$ positions less important than
type $3$ position, and so on. We assume that the organization is monotonic
in that for any position profiles $x=\left( x_{1},...,x_{n}\right) $ and $%
z=\left( z_{1},...,z_{n}\right) $ such that $x\leq z$, $f\left( x\right)
\leq f\left( z\right) $. In any position profile $x=(x_{1},x_{2},..,x_{n})%
\in T^{N}$, $x_{i}$ denotes the position of player $i$.

In a game ladder, players compete for more higher-level positions as
follows. Consider two players $p$ and $q$ in an organization. Player $p$
occupies a position indexed $s$ and player $q$ occupies a position indexed $%
r $, where $s$ is a lower-level position than $r$. We say that $p$ beats $q$
in position $r$ relative to position $s$, denoted $p\succeq _{(r,s)}q$, if
aggregate output increases as a result of swapping $p$ and $q$ no matter how
the other members of the organization are allocated to the different
positions. If $p$ beats $q$ in each position relative to positions that are
lower-ranked, we say that $p$ is globally more influential or productive
than $q$, and this is denoted by $p\succeq q$.

Our goal is to study some mathematical properties of the global influence
relation $\succeq $. We find that this relation is not transitive in
general. However, if it is complete, it is transitive and its asymmetric
component $\succ $ is a finite union of transitive tournaments (that is, $%
\succ $ is a finite unions of complete, asymmetric, and transitive binary
relations).

We note that the global influence relation $\succeq $ coincides with the
replacement relation introduced by Isbell (1958) for games that have only
two layers. However, the properties of $\succeq $ for games which have more
than two layers are quite different the properties of $\succeq $ for games
that have only two layers. In fact, when there are only two layers, $\succeq 
$ is transitive (see, e.g., Taylor and Zwicker (1999)), which contrasts with
the finding we obtain when there are more than two layers.

We also examine whether the relation $\succeq $ predicts the frequency with
which a player is pivotal in the game. The notion of pivotability was
introduced by Freixas (2005b). We show that a more globally influential
player is (weakly) more likely to be pivotal.

The next section shows how the ranking rule $\succeq $ is generated through
the game ladder. Section 3 studies its mathematical properties and concludes.

\section{Ladder Tournament}

A ladder tournament in a game $(N,T,f)$ is a system of challenges that
proceeds as follows. Let $r$ and $s$ be two positions occupied by players $p$
and $q$, respectively, where $r$ is more important than $s$. We say that $p$
beats $q$ in position $r$ relative to position $s$, denoted by $p\succeq
_{(r,s)}q$, if swapping the two players decreases the aggregate output of
the organization no matter how the other players are allocated to the
different positions. We say that $p$ is globally more influential than $q$
if $p\succeq _{(r,s)}q$ for any positions $s$ and $r$ , and this is denoted
by $p\succeq q$. The binary relation $\succeq $ determines the global
ranking of the ladder tournament. The formal definition of each of these
relations is below. We denote by $e^{p}=(0,...,0,1,0,...,0)$ the $p$th unit $%
n$-component vector.

\begin{definition}
Let $\left( N,T,f\right) $ be a game ladder, and $p$ and $q$ two players.

1) Let $s$ and $r$ be two tasks such that $r>s$. We say that $p\succeq
_{(r,s)}q$ if: $\forall x\in T^{N}$ such that $x_{p}=x_{q}=s$, $%
f(x+(r-s)e^{p})\geq f(x+(r-s)e^{q})$.

2) $p$ is said to be more globally influential than $q$, denoted $p\succeq q$%
, if: $p\succeq _{(r,s)}q$ for all $s,r\in T$ such that $r>s$.

The symmetric and asymmetric components of each relation $\succeq _{(r,s)}$
are denoted by $\succ _{(r,s)}$and $\sim _{(r,s)}$, respectively, and those
of the relation $\succeq $ by $\succ $ and $\sim $, respectively.
\end{definition}

We introduce the notion of a linear game ladder.

\begin{definition}
Let $\left( N,T,f\right) $ be a game ladder. $\left( N,T,f\right) $ is said
to be linear if $\succeq $ is complete.
\end{definition}

\section{Properties}

\subsection{Completeness, Transitivity, Tournament}

We show that the symmetric component of the relation $\succeq $ is an
equivalence relation.

\begin{proposition}
Let $\left( N,T,f\right) $ be a game ladder. The relation $\sim $ is an
equivalence relation.
\end{proposition}

\begin{proof}
1) First, let us show that $i\succeq k$.

Let $s,r\in T$ such that $r>s$, and $x\in T^{N}$ such that $x_{i}=x_{k}=s$.
Pose $x_{j}=t$.

We need to show that $f(x+(r-s)e^{k})\leq f(x+(r-s)e^{i})$.

We denote by $x^{\ast }$ the following profile : $x_{u}^{\ast }=\left\{ 
\begin{array}{c}
x_{u}\text{ if }u\neq i,j,k \\ 
0\text{ if }u=i,j,k%
\end{array}%
\right. $

It follows from this notation that $x=x_{u}^{\ast }+se^{i}+se^{k}+te^{j}$
and $\left\{ 
\begin{array}{c}
x+(r-s)e^{k}=x^{\ast }+se^{i}+re^{k}+te^{j} \\ 
x+(r-s)e^{i}=x^{\ast }+se^{k}+re^{i}+te^{j}.%
\end{array}%
\right. $

We distinguish five cases.

\textbf{Case 1}: Assume that $t>r>s$. Then,

\begin{tabular}{lll}
$f(x+(r-s)e^{k})$ & $\leq $ & $f(x^{\ast }+se^{i}+re^{j}+te^{k})$ since $%
k\succeq j$ \\ 
& $\leq $ & $f(x^{\ast }+se^{j}+re^{i}+te^{k})$ since $i\succeq j$ \\ 
& $\leq $ & $f(x^{\ast }+se^{k}+re^{i}+te^{k})$ since $j\succeq k$ \\ 
& $=$ & $f(x+(r-s)e^{i})$.%
\end{tabular}

\textbf{Case 2}: Assume that $t=r$. Then,

\begin{tabular}{lll}
$f(x+(r-s)e^{k})$ & $\leq $ & $f(x^{\ast }+se^{j}+re^{k}+te^{i})$ since $%
i\succeq j$ \\ 
& $\leq $ & $f(x^{\ast }+se^{k}+re^{i}+te^{j})$ since $j\succeq k$ \\ 
& $=$ & $f(x+(r-s)e^{i})$.%
\end{tabular}

\textbf{Case 3}: Assume that $r>t>s$. Then,

\begin{tabular}{lll}
$f(x+(r-s)e^{k})$ & $\leq $ & $f(x^{\ast }+se^{i}+re^{j}+te^{k})$ since $%
j\succeq k$ \\ 
& $\leq $ & $f(x^{\ast }+se^{j}+re^{i}+te^{k})$ since $i\succeq j$ \\ 
& $\leq $ & $f(x^{\ast }+se^{k}+re^{i}+te^{j})$ since $j\succeq k$ \\ 
& $=$ & $f(x+(r-s)e^{i})$.%
\end{tabular}

\textbf{Case 4}: Assume that $t=s$. Then,

\begin{tabular}{lll}
$f(x+(r-s)e^{k})$ & $\leq $ & $f(x^{\ast }+se^{i}+re^{j}+te^{k})$ since $%
j\succeq k$ \\ 
& $\leq $ & $f(x^{\ast }+se^{j}+re^{i}+te^{k})$ since $i\succeq j$ \\ 
& $=$ & $f(x+(r-s)e^{i})$.%
\end{tabular}

\textbf{Case 5}: Assume that $s<t$. Then,

\begin{tabular}{lll}
$f(x+(r-s)e^{k})$ & $\leq $ & $f(x^{\ast }+se^{i}+re^{j}+te^{k})$ since $%
j\succeq k$ \\ 
& $\leq $ & $f(x^{\ast }+se^{j}+re^{i}+te^{k})$ since $k\succeq i$ \\ 
& $\leq $ & $f(x^{\ast }+se^{k}+re^{i}+te^{j})$ since $k\succeq j$ \\ 
& $=$ & $f(x+(r-s)e^{i})$.%
\end{tabular}

2) $k\succeq i$ is proved in the same manner as the proof of $i\succeq k$,
and we conclude that $i\sim k$.
\end{proof}

Our second result says that the relation $\succeq $ is not complete in
general, which implies that $\succeq $ is not a tournament (a tournament is
a binary relation that is complete and asymmetric).

\begin{proposition}
The relation $\succeq $ is not complete in general.
\end{proposition}

\begin{proof}
Consider the following game ladder where $N=\{1,2,3,4,5,6,7\}$, $T=\{1,2,3\}$%
, and $f$ is defined as follows:

For any profile $x$, $f(x)=\left\{ 
\begin{array}{l}
1\text{ if there exists }y\in I\text{ such that }x\leq y \\ 
0\text{ if not}%
\end{array}%
\right. $

where $I=\{(3,1,2,1,1,2,2),(1,3,2,1,1,2,2),(1,2,3,1,1,2,2)\}$.

It can be easily checked that $1$ and $3$ are not comparable.
\end{proof}

Our third result says that the relation $\succeq $ is not transitive in
general.

\begin{proposition}
The relation $\succeq $ is not transitive in general.
\end{proposition}

\begin{proof}
Consider the same game as in the proof of Proposition 2. It can be easily
checked that $1\succ 2$, $2\succ 3$, but $not(1\succ 3)$ given that $1$ and $%
3$ are not comparable.
\end{proof}

Our fourth result states that the global influence relation $\succeq $ is
transitive if $\succeq $ is complete.

\begin{proposition}
Let $\left( N,T,f\right) $ be a linear game ladder. Then $\succeq $ is
transitive.
\end{proposition}

\begin{proof}
Let $\left( N,T,f\right) $ be a linear game ladder, and $i,j$ and $k$ be
three players such that $i\succeq j$ and $j\succeq k.$ We will show that $%
i\succeq k.$ We distinguish two cases.

1) Assume that $not(k\succeq i)$. Then $i\succ k$ since $\left( N,T,f\right) 
$ is linear, and this implies that $i\succeq k$.

2) Assume that $k\succeq i$ : we shall show that in this condition, $%
k\succeq i$.

Let exist $s,r\in T$ such that $r>s$, and $x\in T^{N}$ such that $%
x_{i}=x_{k}=s$. Pose $x_{j}=t$.

We need to show that $f(x+(r-s)e^{k})\leq f(x+(r-s)e^{i})$.

As in Proposition 1, we denote by $x^{\ast }$ the profile such that $%
x_{u}^{\ast }=\left\{ 
\begin{array}{c}
x_{u}\text{ if }u\neq i,j,k \\ 
0\text{ if }u=i,j,k\text{.}%
\end{array}%
\right. $

Then, $x=x_{u}^{\ast }+se^{i}+se^{k}+te^{j}$ and $\left\{ 
\begin{array}{c}
x+(r-s)e^{k}=x^{\ast }+se^{i}+re^{k}+te^{j} \\ 
x+(r-s)e^{i}=x^{\ast }+se^{k}+re^{i}+te^{j}\text{.}%
\end{array}%
\right. $

We shall distinguish five cases as we did in Proposition 1.

\textbf{Case 1}: Assume $t>r>s$. Then,

\begin{tabular}{lll}
$f(x+(r-s)e^{k})$ & $\leq $ & $f(x^{\ast }+se^{j}+re^{k}+te^{i})$ since $%
i\succeq j$ \\ 
& $\leq $ & $f(x^{\ast }+se^{j}+re^{i}+te^{k})$ since $k\succeq i$ \\ 
& $\leq $ & $f(x^{\ast }+se^{k}+re^{i}+te^{j})$ since $j\succeq k$ \\ 
& $=$ & $f(x+(r-s)e^{i})$.%
\end{tabular}

\textbf{Case 2}: Assume $t=r$. The proof is exactly as in Proposition 1.

\textbf{Case 3:} Assume $r<t<s$. The proof is exactly as in Proposition 1.

\textbf{Case 4}: Assume $t=s$. The proof is exactly as in Proposition 1.

\textbf{Case 5}: Assume $s<t$. Then,

\begin{tabular}{lll}
$f(x+(r-s)e^{k})$ & $\leq $ & $f(x^{\ast }+se^{i}+re^{j}+te^{k})$ since $%
j\succeq k$ \\ 
& $\leq $ & $f(x^{\ast }+se^{k}+re^{j}+te^{i})$ since $k\succeq i$ \\ 
& $\leq $ & $f(x^{\ast }+se^{k}+re^{i}+te^{j})$ since $i\succeq j$ \\ 
& $=$ & $f(x+(r-s)e^{i})$.%
\end{tabular}
\end{proof}

Our fifth result states that, if $\succeq $ is complete, then if a player $i$
dominates another player $j$ and $j$ is equivalent to another player $k$,
then $i$ dominates $k$. Also, if a player $i$ is equivalent to another
player $j$ and $j$ dominates another player $k$, then $i$ dominates $k$.

\begin{proposition}
Let $\left( N,T,f\right) $ be a linear game ladder. If $i\succ j$ and $j\sim
k$, then $i\succ k$. Also, if $i\sim j$ and $j\succ k$, then $i\succ k$.
\end{proposition}

\begin{proof}
Let $\left( N,T,f\right) $ be a linear game ladder. Let $i,j$ and $k$ be
three players such that $i\succ j$ and $j\sim k$. It follows from what
precedes that if $not(i\succ k)$ then $k\succeq i$. But $j\succeq k$ and $%
k\succeq i$ imply $j\succeq i$, which is a contradiction with $i\succ j$.

The second statement is proved similarly.
\end{proof}

Our sixth result states that, if $\succeq $ is complete, then if a player $i$
dominates another player $j$ and $j$ dominates another player $k$, then $i$
dominates $k$.

\begin{proposition}
Let $\left( N,T,f\right) $ be a linear game ladder. Then, $\succ $ is
transitve.
\end{proposition}

\begin{proof}
Assume that $i\succ j$, $j\succ k$ and $not(i\succ k)$. Then $k\succeq i$.
Since $k\succeq i$ and $i\succ j$, it follows that $k\succ j$, which
contradicts $j\succ k$.
\end{proof}

It follows from Propositions 1-6 that $\succeq $ is not a transitive
tournament in general. However, if $\succeq $ is complete, then it is a
complete preorder and its asymmetric component $\succ $ is transitive. This
implies that, if $\succeq $ is complete, then players can be partitioned
into a finite number of layers where players in the same layer are
equivalent and players in each layer dominate those in lower-level layers.
The results also imply that, if $\succeq $ is complete, then $\succ $ is a
finite union of transitive tournaments.

\begin{theorem}
Let $\left( N,T,f\right) $ be a linear game ladder. Then $\succeq $ is a
complete preorder and $\succ $ is a finite union of tournaments.
\end{theorem}

\begin{proof}
The proof simply follows from the proofs of Propositions 1-6.
\end{proof}

\subsection{Tournament Rank and Pivotability}

In this section, we study the relationship between tournament rank, given by
the relation $\succeq $, and pivotability as defined by Freixas (2005a). Let 
$\left( N,T,f\right) $ be a game ladder, and $\{z_{1},z_{2},...,z_{k}\}$ the
range of $f$. Without loss of generality, we assume $z_{1}>z_{2}>...>z_{k}$.
Denote by $Q_{N}$ the set of all the bijections defined from $N$ onto $%
\left\{ 1,...,n\right\} $. An element of $Q_{N}$ is a possible order in
which players entire the game or the organization (players move from a type $%
1$ task to another position type or they stay in $1$). An ordered allocation
of positions of $N\ $is an ordered pair $R=\left( sR,x^{{\small R}}\right) ,$
where $sR\in Q_{N}$ and $x^{{\small R}}\in T^{N}$.

The $i$-pivotal player of $R$ for $i=1,$ $2,$ $...,$ $k-1$, denoted by $i$%
-piv$\left( R,f\right) $, is uniquely defined either as:

(a): the player whose action in $R$ clinches the outcome of $x^{R}$ under at
least the output level $z_{i}$, or

(b): the player whose action in $R$ clinches the outcome of $x^{R}$ under at
most the output level $z_{i+1}$.

In other words, $i$-piv$\left( R,f\right) $ is the first player in $sR$ who
satisfies one of the two following mutually exclusive conditions:

(1) independently of how tasks are allocated among all subsequent players,
the outcome will be $f_{h}$ with $f_{h}\geq f_{i}$, or

(2) no matter how all subsequent players were to change their actions, the
final outcome would be no greater than $f_{i+1}$.

For every $i\in \{1,$ $2,$ $...,$ $k-1\}$, denote by $\mathcal{R}%
_{ip}^{+}\left( f\right) $ the set of all the ordered allocation of
positionss for which player $p$ is $i$-pivotal.

We answer the question of whether the frequency with which a player is
pivotal reflects tournament rank. I

\begin{theorem}
Let $\left( N,T,f\right) $ be a game ladder, and $p,q$ $\in N$ two players.
Then, $p\succeq q\Longrightarrow \left\vert \mathcal{R}_{ip}^{+}\right\vert
\geq \left\vert \mathcal{R}_{iq}^{+}\right\vert $ for any $i\in \left\{ 1,%
\text{ }.\text{ }.\text{ }.,\text{ }k-1\right\} $.
\end{theorem}

The proof of Theorem 2 is quite involved and is presented in Section 3.3.
This theorem proves that a player's pivotability is a weakly increasing
function of his tournament rank, which justifies the use of tournament in
organizations.

\subsection{\textbf{Proof of Theorem 2}}

\bigskip To prove Theorem 2, a few definitions and five preliminary results
will be needed.

\subsubsection{Preliminary Definitions and Results for the Proof of Theorem 2%
}

Let $R=\left( sR,\text{ }x^{{\small R}}\right) $ be an ordered allocation of
positions. Let $p$ be a player. $R_{p}^{j}=(sR_{p}^{j},$ $x^{{\small R}_{%
{\small p}}^{{\small j}}})$ is the ordered allocation of positions in which
the order in which players enter the organization coincides with that of $R$
(i.e. $sR_{p}^{j}=sR$) and player $p$ as well as any player $q$ who enters
before $p$ occupies the exact same position as in $R$ (i.e. $x_{q}^{{\small R%
}_{{\small p}}^{{\small j}}}=x_{q}^{{\small R}}$\ if $sR\left( q\right) \leq
sR\left( p\right) $) whereas players who enter after $p$ are assigned the
highest position (i.e. $x_{q}^{{\small R}_{{\small p}}^{{\small j}}}=j$\ if $%
sR\left( q\right) >sR\left( p\right) $). Similarly, $R_{p}^{1}=\left(
sR_{p}^{1},x^{{\small R}_{{\small p}}^{{\small 1}}}\right) $ is the ordered
allocation of positions in which the order in which players enter the
organization coincides with that of $R$ (i.e. $sR_{p}^{1}=sR$) and player $p$
as well as any player $q$ who enters before $p$ occupies the exact same
position as in $R$ (i.e. $x_{q}^{{\small R}_{{\small p}}^{1}}=x_{q}^{{\small %
R}}$\ if $sR\left( q\right) \leq sR\left( p\right) $) whereas players who
enter after $p$ are assigned the lowest position (i.e. $x_{q}^{{\small R}_{%
{\small p}}^{1}}=1$\ if $sR\left( q\right) >sR\left( p\right) $). These
definitions are formalized in Notation 1 below, where we also define the set
of all ordered allocation of positionss in which a player is pivotal.

\begin{notation}
Let $\left( N,T,f\right) $ be a game ladder, $R=\left( sR,\text{ }%
x^{R}\right) $ an ordered allocation of positions, and $p$ a player.

$i)$ We denote by $R_{p}^{j}=(sR_{p}^{j},$ $x^{{\small R}_{{\small p}}^{%
{\small j}}})$ the ordered allocation of positions defined by:

$sR_{p}^{j}=sR,$ and $\forall q\in N,$ $x_{q}^{{\small R}_{{\small p}}^{%
{\small j}}}=\left\{ 
\begin{array}{l}
x_{q}^{{\small R}}\text{ \ \ \ \ \ \ \ \ if }sR\left( q\right) \leq sR\left(
p\right) \\ 
j\text{ \ \ \ \ \ \ \ \ \ \ if }sR\left( q\right) >sR\left( p\right) \text{.}%
\end{array}%
\right. $

$ii)$ We denote by $R_{p}^{1}=\left( sR_{p}^{1},x^{{\small R}_{{\small p}}^{%
{\small 1}}}\right) $ the ordered allocation of positions defined by:

$sR_{p}^{1}=sR,$ and for all $q\in N$, $x_{q}^{{\small R}_{{\small p}}^{%
{\small 1}}}=\left\{ 
\begin{array}{l}
x_{q}^{{\small R}}\text{ \ \ \ \ \ \ \ \ \ if }sR\left( q\right) \leq
sR\left( p\right) \\ 
1\text{\ \ \ \ \ \ \ \ \ \ \ \ if }sR\left( q\right) >sR\left( p\right) 
\text{.}%
\end{array}%
\right. $

$iii)$ For any $i=1,$ $2,$ $.$ $.$ $.,$ $k-1$, $\mathcal{R}_{ip}^{+}\left(
f\right) $ denotes the set of all the ordered allocation of positionss for
which player $p$ is $i$-pivotal. We pose: 
\begin{equation*}
\mathcal{R}_{p}\left( f\right) =\dbigcup\limits_{i=1}^{k-1}\text{ }\mathcal{R%
}_{ip}^{+}\left( f\right) \text{.}
\end{equation*}
\end{notation}

\bigskip

We also define the notion of agreement between two ordered allocations of
positions. Let $R=(sR,x^{R})$ and $R^{\prime }=(sR^{\prime },x^{R^{\prime }})
$ two ordered allocations of positions in a game ladder $(N,T,f)$, and let $p
$ be a player. We say that $R$ and $R^{\prime }$

agree up to the player $p$ if the order in which the players enter is the
same for $R$ and $R^{\prime }$ (i.e. $sR=sR^{\prime }$) and $p$ as well as
any player who enters before him occupies the same position in $R$ as in $%
R^{\prime }$(i.e. $x_{q}^{{\small R}_{{\small p}}^{{\small j}}}=x_{q}^{%
{\small R}}$\ if $sR\left( q\right) \leq sR\left( p\right) $). Formally, $R$
and $R^{\prime }$ agree up to the player $p$ if $\left\{ 
\begin{array}{l}
sR=sR^{\prime }\text{\ and }q\leq sR\left( p\right)  \\ 
x_{q}^{R}=x_{p}^{R^{\prime }}\text{ for any }q\text{\ such that }sR\left(
q\right) \leq sR\left( p\right) \text{.}%
\end{array}%
\right. $

We denote by $\mathcal{A}\left( R,p\right) $ the set of all ordered
allocations of positions which agree with $R$ up to $p$.

The following remark is important.

\begin{remark}
Let $R$ be an ordered allocation of positions. $\forall R^{\prime }\in 
\mathcal{A}\left( R,p\right) $, $x^{{\small R}_{{\small p}}^{{\small j}%
}}\leq x^{{\small R}^{\prime }}$ and $x^{{\small R}^{\prime }}\leq x^{%
{\small R}_{{\small p}}^{{\small 1}}}$.
\end{remark}

Denote by $prec\left( p,sR\right) $ the player who enters the organization
right before player $p$ in an ordered allocation of positions $R$. We remark
that the definition of a $i$-pivotal player can be stated as follows:

A player $p$ is $i$-pivotal in $R$ if:

\begin{equation*}
\left\{ 
\begin{array}{l}
1)\text{ }\forall R^{\prime }\in \mathcal{A}\left( R,p\right) ,\text{ }%
f\left( x^{{\small R}^{\prime }}\right) \geq z_{i} \\ 
2)\text{ }sR\left( p\right) \neq 1\Rightarrow \exists R^{\prime }\in 
\mathcal{A}\left( R,prec\left( p,sR\right) \right) :f\left( x^{{\small R}%
^{\prime }}\right) <z_{i}%
\end{array}%
\right.
\end{equation*}%
or 
\begin{equation*}
\left\{ 
\begin{array}{l}
1)\text{ }\forall R^{\prime }\in \mathcal{A}\left( R,p\right) ,\text{ }%
f\left( x^{{\small R}^{\prime }}\right) <z_{i} \\ 
2)\text{ }sR\left( p\right) \neq 1\Rightarrow \exists R^{\prime }\in 
\mathcal{A}\left( R,\text{ }prec\left( p,sR\right) \right) :f\left( x^{%
{\small R}^{\prime }}\right) \geq z_{i}.%
\end{array}%
\right.
\end{equation*}

\bigskip

The following lemma provides a full characterization of $\mathcal{R}%
_{ip}^{+}\left( f\right) $.

\begin{lemma}
Let $\left( N,T,f\right) $ be a game ladder, $p$ a player, $R$ an ordered
allocation of positions, and $i\in \left\{ 1,2,...,k-1\right\} .$%
\begin{equation*}
R\in \text{ }\mathcal{R}_{ip}^{+}\left( f\right) \Longleftrightarrow \left\{ 
\begin{array}{l}
f(x^{{\small R}_{{\small p}}^{{\small j}}})\geq z_{i} \\ 
f\left( x^{{\small R}_{{\small p}}^{{\small j}}}+(j-x_{p}^{{\small R}_{%
{\small p}}^{{\small j}}})e^{p}\right) <z_{i}%
\end{array}%
\right. \text{ or }\left\{ 
\begin{array}{l}
f\left( x^{{\small R}_{{\small p}}^{{\small 1}}}\right) <z_{i} \\ 
f\left( x^{{\small R}_{{\small p}}^{{\small 1}}}+(1-x_{p}^{{\small R}_{%
{\small p}}^{{\small 1}}})e^{p}\right) \geq z_{i}\text{.}%
\end{array}%
\right.
\end{equation*}
\end{lemma}

\begin{proof}
$\Longrightarrow )$ Suppose that $R\in $ $\mathcal{R}_{ip}^{+}\left(
f\right) $ and let us show the following: 
\begin{equation*}
\left\{ 
\begin{array}{l}
f(x^{{\small R}_{{\small p}}^{{\small j}}})\geq z_{i} \\ 
f\left( x^{{\small R}_{{\small p}}^{{\small j}}}+(j-x_{p}^{{\small R}_{%
{\small p}}^{{\small j}}})e^{p}\right) <z_{i}%
\end{array}%
\right. \text{ or }\left\{ 
\begin{array}{l}
f\left( x^{{\small R}_{{\small p}}^{{\small 1}}}\right) <z_{i} \\ 
f\left( x^{{\small R}_{{\small p}}^{{\small 1}}}+(1-x_{p}^{{\small R}_{%
{\small p}}^{{\small 1}}})e^{p}\right) \geq z_{i}%
\end{array}%
\right.
\end{equation*}%
Two cases are possible : $sR\left( p\right) \neq 1$ and $sR\left( p\right)
=1 $

\textbf{Case 1} : $sR\left( p\right) \neq 1.$

\begin{itemize}
\item If $\left\{ 
\begin{array}{l}
\forall R^{\prime }\in \mathcal{A}(R,p),\text{ }f\left( x^{{\small R}%
^{\prime }}\right) \geq z_{i} \\ 
\exists R^{\prime }\in \mathcal{A}(R,\text{ }prec(p,sR)):f\left( x^{{\small R%
}^{\prime }}\right) <z_{i}%
\end{array}%
\right. $ then, given that $R_{p}^{j}\in \mathcal{A}\left( R,p\right) $, we
have 
\begin{equation*}
f\left( x^{{\small R}_{{\small p}}^{{\small j}}}\right) \geq z_{i}\text{.}
\end{equation*}
\end{itemize}

Let the ordered allocation of positions $R_{1}=\left( sR_{1},x^{{\small R}_{%
{\small 1}}}\right) $ be such that 
\begin{equation*}
sR_{1}=sR
\end{equation*}

and%
\begin{equation*}
x^{{\small R}_{{\small 1}}}=x^{{\small R}_{{\small p}}^{{\small j}%
}}+(j-x_{p}^{{\small R}_{{\small p}}^{{\small j}}})e^{p}\text{.}
\end{equation*}

We have $sR_{1}=sR=sR^{\prime }$. In addition, $x^{{\small R}_{{\small 1}%
}}\leq x^{{\small R}^{\prime }}$ by definition of $R_{1}$ and $R^{\prime }$.
Therefore, \newline
\begin{equation*}
f(x^{{\small R}_{{\small p}}^{{\small j}}}+(j-x_{p}^{{\small R}_{{\small p}%
}^{{\small j}}})e^{p})\leq f(x^{{\small R}^{\prime }})<z_{i}.
\end{equation*}

\begin{itemize}
\item If$\left\{ 
\begin{array}{l}
\forall R^{\prime }\in \mathcal{A}\left( R,p\right) ,\text{ }f\left( x^{%
{\small R}^{\prime }}\right) <z_{i} \\ 
\exists R^{\prime }\in \mathcal{A}\left( R,prec\left( p,sR\right) \right)
:f\left( x^{{\small R}^{\prime }}\right) \geq z_{i}%
\end{array}%
\right. $then, given that $R_{p}^{1}\in \mathcal{A}\left( R,p\right) ,$ we
have
\end{itemize}

\begin{equation*}
f\left( x^{{\small R}_{{\small p}}^{{\small 1}}}\right) <z_{i}.
\end{equation*}

Let the ordered allocation of positions $R_{2}=\left( sR_{2},x^{{\small R}_{%
{\small 2}}}\right) $ be such that 
\begin{equation*}
sR_{2}=sR
\end{equation*}

and%
\begin{equation*}
x^{{\small R}_{{\small 2}}}=x^{{\small R}_{{\small p}}^{{\small 1}%
}}+(1-x_{p}^{{\small R}_{{\small p}}^{{\small 1}}})e^{p}\text{.}
\end{equation*}

We have, by definition of $R_{2}$ and $R^{\prime }$ 
\begin{equation*}
sR_{2}=sR=sR^{\prime }\text{ and }x^{{\small R}^{\prime }}\leq x^{{\small R}%
_{{\small 2}}}\text{.}
\end{equation*}

It follows that%
\begin{equation*}
f(x^{{\small R}_{{\small p}}^{{\small 1}}}+(1-x_{p}^{{\small R}_{{\small p}%
}^{{\small 1}}})e^{p})\geq f(x^{{\small R}^{\prime }})\geq z_{i}\text{.}
\end{equation*}%
\medskip

\textbf{Case 2} : $sR\left( p\right) =1$.

\begin{itemize}
\item If $\forall R^{\prime }\in \mathcal{A}\left( R,p\right) ,$ $f\left( x^{%
{\small R}^{\prime }}\right) \geq z_{i}$, then $R_{p}^{j}\in \mathcal{A}%
\left( R,p\right) ,$ and we have
\end{itemize}

\begin{equation*}
f(x^{{\small R}_{{\small p}}^{{\small j}}})\geq z\text{.}
\end{equation*}

We know that there exists a task profile $x$ such that $f\left( x\right)
=z_{k}$.

Suppose that $f(x^{{\small R}_{{\small p}}^{{\small j}}}+(j-x_{p}^{{\small R}%
_{{\small p}}^{{\small j}}})e^{p})\geq z_{i}$. Given that for any $y\in
T^{N} $, the following inequality holds 
\begin{equation*}
x^{{\small R}_{{\small p}}^{{\small j}}}+(j-x_{p}^{{\small R}_{{\small p}}^{%
{\small j}}})e^{p}\leq y
\end{equation*}

it follows that

\begin{equation*}
x^{{\small R}_{{\small p}}^{{\small j}}}+(j-x_{p}^{{\small R}_{{\small p}}^{%
{\small j}}})e^{p}\leq x
\end{equation*}

and by monotonicity, we have 
\begin{equation*}
f(x^{{\small R}_{{\small p}}^{{\small j}}}+(j-x_{p}^{{\small R}_{{\small p}%
}^{{\small j}}})e^{p})\leq f\left( x\right) =z_{k}\text{.}
\end{equation*}%
Therefore, $z_{k}\geq z_{i}$, which is contradictory since $z_{i}>z_{k}$. We
conclude that 
\begin{equation*}
f(x^{{\small R}_{{\small p}}^{{\small j}}}+(j-x_{p}^{{\small R}_{{\small p}%
}^{{\small j}}})e^{p})<z_{i}.
\end{equation*}

\begin{itemize}
\item If $\forall R^{\prime }\in \mathcal{A}\left( R,p\right) $, $f\left(
x^{R^{\prime }}\right) <z_{i}$, then, given that $R_{p}^{1}\in \mathcal{A}%
\left( R,p\right) ,$ we have
\end{itemize}

\begin{equation*}
f\left( x^{{\small R}_{{\small p}}^{{\small 1}}}\right) <z_{i}\text{.}
\end{equation*}

We know that there exists a task profile $x$ such that $f\left( x\right)
=z_{1}$.

Suppose that $f(x^{{\small R}_{{\small p}}^{{\small 1}}}+(1-x_{p}^{{\small R}%
_{{\small p}}^{{\small 1}}})e^{p})<z_{i}$. Then, given that for any $y\in
T^{N}$, the following inequality holds 
\begin{equation*}
y\leq x^{{\small R}_{{\small p}}^{{\small 1}}}+(1-x_{p}^{{\small R}_{{\small %
p}}^{{\small 1}}})e^{p}
\end{equation*}

we have 
\begin{equation*}
x\leq x^{{\small R}_{{\small p}}^{{\small 1}}}+(1-x_{p}^{{\small R}_{{\small %
p}}^{{\small 1}}})e^{p}
\end{equation*}

and by monotonicity, 
\begin{equation*}
z_{1}=f\left( x\right) \leq f(x^{{\small R}_{{\small p}}^{{\small 1}%
}}+(1-x_{p}^{{\small R}_{{\small p}}^{{\small 1}}})e^{p})\text{.}
\end{equation*}%
That is, 
\begin{equation*}
f(x^{{\small R}_{{\small p}}^{{\small 1}}}+(j-x_{p}^{{\small R}_{{\small p}%
}^{{\small 1}}})e^{p})\geq z_{1}\text{.}
\end{equation*}

And therefore, we have $z_{1}<z_{i}$, which is contradictory since $%
z_{1}\geq z_{i}$. Hence 
\begin{equation*}
f(x^{{\small R}_{{\small p}}^{{\small 1}}}+(j-x_{p}^{{\small R}_{{\small p}%
}^{{\small 1}}})e^{p})\geq z_{i}\text{.}
\end{equation*}

We conclude that 
\begin{equation*}
\left\{ 
\begin{array}{l}
f(x^{{\small R}_{{\small p}}^{{\small j}}})\geq z_{i} \\ 
f(x^{{\small R}_{{\small p}}^{{\small j}}}+(j-x_{p}^{{\small R}_{{\small p}%
}^{{\small j}}})e^{p})<z_{i}%
\end{array}%
\right. \text{ or }\left\{ 
\begin{array}{l}
f\left( x^{{\small R}_{{\small p}}^{{\small 1}}}\right) <z_{i} \\ 
f\left( x^{{\small R}_{{\small p}}^{{\small 1}}}+(j-x_{p}^{{\small R}_{%
{\small p}}^{{\small 1}}})e^{p}\right) \geq z_{i}\text{.}%
\end{array}%
\right.
\end{equation*}

$\Longleftarrow )$ Conversely, assume $\left\{ 
\begin{array}{l}
f(x^{{\small R}_{{\small p}}^{{\small j}}})\geq z_{i} \\ 
f(x^{{\small R}_{{\small p}}^{{\small j}}}+(j-x_{p}^{{\small R}_{{\small p}%
}^{{\small j}}})e^{p})<z_{i}%
\end{array}%
\right. $ or $\left\{ 
\begin{array}{l}
f\left( x^{{\small R}_{{\small p}}^{{\small 1}}}\right) <z_{i} \\ 
f\left( x^{{\small R}_{{\small p}}^{{\small 1}}}+(j-x_{p}^{{\small R}_{%
{\small p}}^{{\small 1}}})e^{p}\right) \geq z_{i}%
\end{array}%
\right. $ and let us show that $R\in $ $\mathcal{R}_{ip}^{+}\left( f\right) $%
.

We shall consider two cases : $sR\left( p\right) \neq 1$ and $sR\left(
p\right) =1$.

\medskip

\textbf{Case 1}: $sR\left( p\right) \neq 1.$

Let us show that 
\begin{equation*}
\left\{ 
\begin{array}{l}
\forall R^{\prime }\in \mathcal{A}\left( R,p\right) ,\text{ }f\left(
x^{R^{\prime }}\right) \geq z_{i} \\ 
\exists R^{\prime }\in \mathcal{A}\left( R,prec\left( p,sR\right) \right)
:f\left( x^{R^{\prime }}\right) <z_{i}%
\end{array}%
\right.
\end{equation*}%
or \ 
\begin{equation*}
\left\{ 
\begin{array}{l}
\forall R^{\prime }\in \mathcal{A}\left( R,p\right) ,\text{ }f\left(
x^{R^{\prime }}\right) <z_{i} \\ 
\exists R^{\prime }\in \mathcal{A}\left( R,prec\left( p,sR\right) \right)
:f\left( x^{R^{\prime }}\right) \geq z_{i}%
\end{array}%
\right.
\end{equation*}

\begin{itemize}
\item If $\left\{ 
\begin{array}{l}
f(x^{{\small R}_{{\small p}}^{{\small j}}})\geq z_{i} \\ 
f(x^{{\small R}_{{\small p}}^{{\small j}}}+(j-x_{p}^{{\small R}_{{\small p}%
}^{{\small j}}})e^{p})<z_{i}%
\end{array}%
\right. $, then, since for all $R^{\prime }\in \mathcal{A}\left( R,p\right) $%
, we have
\end{itemize}

\ 
\begin{equation*}
x^{{\small R}_{{\small p}}^{{\small j}}}\leq x^{R^{\prime }}andf(x^{{\small R%
}_{{\small p}}^{{\small j}}})\geq z_{i}
\end{equation*}

and by monotonicity, 
\begin{equation*}
f\left( x^{R^{\prime }}\right) \geq z_{i},\text{ }\forall R^{\prime }\in 
\mathcal{A}\left( R,p\right) .
\end{equation*}

Pose $R^{\prime }=(sR^{\prime },x^{R^{\prime }})$ with $sR^{\prime }=sR$ and 
$x^{R^{\prime }}=x^{{\small R}_{{\small p}}^{{\small j}}}+(j-x_{p}^{{\small R%
}_{{\small p}}^{{\small j}}})e^{p}$.

We have 
\begin{equation*}
R^{\prime }\in \mathcal{A}\left( R,prec\left( p,sR\right) \right) \text{ and 
}f(x^{{\small R}_{{\small p}}^{{\small j}}}+(j-x_{p}^{{\small R}_{{\small p}%
}^{{\small j}}})e^{p})<z_{i}\text{.}
\end{equation*}

Hence 
\begin{equation*}
\left\{ 
\begin{array}{l}
\forall R^{\prime }\in \mathcal{A}\left( R,p\right) ,\text{ }f\left(
x^{R^{\prime }}\right) \geq z_{i} \\ 
\exists R^{\prime }\in \mathcal{A}\left( R,prec\left( p,sR\right) \right)
:f\left( x^{R^{\prime }}\right) <z_{i}\text{.}%
\end{array}%
\right.
\end{equation*}%
$.$

\begin{itemize}
\item If $\left\{ 
\begin{array}{l}
f\left( x^{{\small R}_{{\small p}}^{{\small 1}}}\right) <z_{i} \\ 
f(x^{{\small R}_{{\small p}}^{{\small 1}}}+(1-x_{p}^{{\small R}_{{\small p}%
}^{{\small 1}}})e^{p})\geq z_{i}%
\end{array}%
\right. $, then since for all $R^{\prime }\in \mathcal{A}\left( R,p\right) $%
, we have , it follows by monotonicity that .
\end{itemize}

\begin{equation*}
x^{R^{\prime }}\leq x^{{\small R}_{{\small p}}^{{\small 1}}}\text{ and }%
f\left( x^{{\small R}_{{\small p}}^{{\small 1}}}\right) <z_{i}
\end{equation*}

it follows by monotonicity that 
\begin{equation*}
f\left( x^{R^{\prime }}\right) <z_{i},\forall R^{\prime }\in \mathcal{A}%
\left( R,p\right) .
\end{equation*}

Let us find $R^{\prime }\in \mathcal{A}\left( R,\text{ }prec\left(
p,sR\right) \right) $ such that 
\begin{equation*}
f\left( x^{R^{\prime }}\right) \geq z_{i}.
\end{equation*}

It suffices to consider $R^{\prime }=(sR^{\prime },x^{R^{\prime }})$ with $%
sR^{\prime }=sR$ and $x^{R^{\prime }}=x^{{\small R}_{{\small p}}^{{\small 1}%
}}+(1-x_{p}^{{\small R}_{{\small p}}^{{\small 1}}})e^{p}$. This ends the
proof for \textbf{Case 1}.

\textbf{Case 2}: $sR\left( p\right) =1.$

Let us show that $\forall R^{\prime }\in \mathcal{A}\left( R,p\right) ,$ $%
f\left( x^{R^{\prime }}\right) \geq z_{i}$ or $\forall R^{\prime }\in 
\mathcal{A}\left( R,p\right) ,$ $f\left( x^{R^{\prime }}\right) <z_{i}$.

\begin{itemize}
\item If $\left\{ 
\begin{array}{l}
f(x^{{\small R}_{{\small p}}^{{\small j}}})\geq z_{i} \\ 
f\left( x^{{\small R}_{{\small p}}^{{\small j}}}+(j-x_{p}^{{\small R}_{%
{\small p}}^{{\small j}}})e^{p}\right) <z_{i}%
\end{array}%
\right. ,$ since for all $R^{\prime }\in \mathcal{A}\left( R,p\right) $, we
have
\end{itemize}

\begin{equation*}
x^{{\small R}_{{\small p}}^{{\small j}}}\leq x^{R^{\prime }}\text{ and }f(x^{%
{\small R}_{{\small p}}^{{\small j}}})\geq z_{i}\text{.}
\end{equation*}

It follows by monotonicity that: 
\begin{equation*}
f\left( x^{R^{\prime }}\right) \geq z_{i},\text{ }\forall R^{\prime }\in 
\mathcal{A}\left( R,p\right) .
\end{equation*}

\begin{itemize}
\item If $\left\{ 
\begin{array}{l}
f\left( x^{{\small R}_{{\small p}}^{{\small 1}}}\right) <z_{i} \\ 
f\left( x^{{\small R}_{{\small p}}^{{\small 1}}}+(1-x_{p}^{{\small R}_{%
{\small p}}^{{\small 1}}})e^{p}\right) \geq z_{i}%
\end{array}%
\right. $, since for all $R^{\prime }\in \mathcal{A}\left( R,p\right) $, we
have
\end{itemize}

\begin{equation*}
x^{R^{\prime }}\leq x^{{\small R}_{{\small p}}^{{\small 1}}}\text{ and }%
f\left( x^{{\small R}_{{\small p}}^{{\small 1}}}\right) <z_{i}\text{.}
\end{equation*}

It follows by monotonicity that 
\begin{equation*}
f\left( x^{R^{\prime }}\right) <z_{i},\text{ }\forall R^{\prime }\in 
\mathcal{A}\left( R,p\right) \text{.}
\end{equation*}%
We deduce that for all $R^{\prime }\in \mathcal{A}\left( R,p\right) $, 
\begin{equation*}
f\left( x^{R^{\prime }}\right) \geq z_{i}\text{ or }f\left( x^{R^{\prime
}}\right) <z_{i}\text{.}
\end{equation*}

We therefore conclude that $R\in $ $\mathcal{R}_{ip}^{+}\left( f\right) $.
\end{proof}

\bigskip

The following notation is useful.

\begin{notation}
Let $R=\left( sR,x^{{\small R}}\right) $ be an ordered allocation of
positions. We denote respectively by $R_{pq}$ and $R_{pq}^{0}$ the following
ordered allocation of positionss:

1) $\left\{ 
\begin{array}{l}
sR_{pq}\left( p\right) =sR\left( q\right) ;\text{ }sR_{pq}\left( q\right)
=sR\left( p\right) \text{ and }sR_{pq}\left( a\right) =sR\left( a\right) 
\text{ for all }a\notin \left\{ p,q\right\}  \\ 
\text{and} \\ 
x^{{\small R}_{{\small pq}}}=x^{{\small R}}+(x_{q}^{{\small R}}-x_{p}^{%
{\small R}})e^{p}+(x_{p}^{{\small R}}-x_{q}^{{\small R}})e^{q}%
\end{array}%
\right. $

2) $\left\{ 
\begin{array}{l}
sR_{pq}^{0}\left( p\right) =sR\left( q\right) ;sR_{pq}^{0}\left( q\right)
=sR\left( p\right) \text{ and }sR_{pq}^{0}\left( a\right) =sR\left( a\right) 
\text{ for all }a\notin \left\{ p,q\right\} \\ 
\text{and} \\ 
x^{{\small R}_{{\small pq}}^{{\small 0}}}=x^{{\small R}}%
\end{array}%
\right. $
\end{notation}

\bigskip

We have the following remark.

\begin{remark}
Let $R=\left( sR,x^{{\small R}}\right) $ be an ordered allocation of
positions, $p,q\in N:x_{{\small p}}^{{\small R}}=r$ and $s=x_{{\small q}}^{%
{\small R}}$.

$1)$ If $sR\left( p\right) <sR\left( q\right) $, then,%
\begin{eqnarray*}
x^{{\small R}_{{\small pq}}^{{}}} &=&x^{{\small R}}+(x_{q}^{{\small R}%
}-x_{p}^{{\small R}})e^{p}+(x_{p}^{{\small R}}-x_{q}^{{\small R}})e^{q} \\
x^{{\small R}_{{\small q}}^{{\small j}}} &=&x^{\left( {\small R}_{{\small pq}%
}^{{\small 0}}\right) _{{\small p}}^{{\small j}}}\text{ } \\
x^{{\small R}_{{\small q}}^{{\small 1}}} &=&x^{\left( {\small R}_{{\small pq}%
}^{{\small 0}}\right) _{{\small p}}^{{\small 1}}}\text{.}
\end{eqnarray*}%
$2)$ If $sR\left( p\right) <sR\left( q\right) $, then,%
\begin{eqnarray*}
x^{{\small R}_{{\small q}}^{{\small j}}}+(x_{q}^{{\small R}_{{\small q}}^{%
{\small j}}}-x_{p}^{{\small R}_{{\small q}}^{{\small j}}})e^{p}+(x_{p}^{%
{\small R}_{{\small q}}^{{\small j}}}-x_{q}^{{\small R}_{{\small q}}^{%
{\small j}}})e^{q} &=&x^{\left( {\small R}_{{\small pq}}\right) _{{\small p}%
}^{{\small j}}} \\
x^{{\small R}_{{\small q}}^{{\small j}}}+(r-x_{q}^{{\small R}_{{\small q}}^{%
{\small j}}})e^{q}+(j-x_{p}^{{\small R}_{{\small q}}^{{\small j}}})e^{p}
&=&x^{\left( {\small R}_{{\small pq}}\right) _{{\small p}}^{{\small j}%
}}+(j-x_{p}^{\left( {\small R}_{{\small pq}}\right) _{{\small p}}^{{\small j}%
}})e^{p} \\
x^{{\small R}_{{\small q}}^{{\small 1}}}+(x_{q}^{{\small R}_{{\small q}}^{%
{\small 1}}}-x_{p}^{{\small R}_{{\small q}}^{{\small 1}}})e^{p}+(x_{p}^{%
{\small R}_{{\small q}}^{{\small 1}}}-x_{q}^{{\small R}_{{\small q}}^{%
{\small 1}}})e^{q} &=&x^{\left( {\small R}_{{\small pq}}\right) _{{\small p}%
}^{{\small 1}}} \\
x^{{\small R}_{{\small q}}^{{\small 1}}}+(r-x_{q}^{{\small R}_{{\small q}}^{%
{\small 1}}})e^{q}+(1-x_{p}^{{\small R}_{{\small q}}^{{\small 1}}})e^{p}
&=&x^{\left( {\small R}_{{\small pq}}\right) _{{\small p}}^{{\small 1}%
}}+(1-x_{p}^{\left( {\small R}_{{\small pq}}\right) _{{\small p}}^{{\small 1}%
}})e^{p}\text{.}
\end{eqnarray*}

$3)$ If $sR\left( p\right) >sR\left( q\right) $, then%
\begin{eqnarray*}
x^{\left( {\small R}_{{\small pq}}\right) _{{\small p}}^{{\small j}%
}}+(j-x_{p}^{\left( {\small R}_{{\small pq}}\right) _{{\small p}}^{{\small j}%
}})e^{p} &=&x^{{\small R}_{{\small q}}^{{\small j}}}+(j-x_{q}^{{\small R}_{%
{\small q}}^{{\small j}}})e^{q} \\
x^{{\small R}_{{\small q}}^{{\small j}}}+(j-x_{q}^{{\small R}_{{\small q}}^{%
{\small j}}})e^{q}+(s-x_{p}^{{\small R}_{{\small q}}^{{\small j}}})e^{p}
&=&x^{\left( {\small R}_{{\small pq}}\right) _{{\small p}}^{{\small j}}} \\
x^{\left( {\small R}_{{\small pq}}\right) _{{\small p}}^{{\small 1}%
}}+(1-x_{p}^{\left( {\small R}_{{\small pq}}\right) _{{\small p}}^{{\small 1}%
}})e^{p} &=&x^{{\small R}_{{\small q}}^{{\small 1}}}+(1-x_{q}^{{\small R}_{%
{\small q}}^{{\small 1}}})e^{q} \\
x^{{\small R}_{{\small q}}^{{\small 1}}}+(1-x_{q}^{{\small R}_{{\small q}}^{%
{\small 1}}})e^{q}+(s-x_{p}^{{\small R}_{{\small q}}^{{\small 1}}})e^{p}
&=&x^{\left( {\small R}_{{\small pq}}\right) _{{\small p}}^{{\small 1}}}%
\text{.}
\end{eqnarray*}
\end{remark}

\bigskip

The following lemma is straightforward.

\begin{lemma}
Let $\left( N,T\text{ }f\right) $ be a game ladder, and $p,q\in N$ two
players. If $R\ $and $R^{\prime }$ are two ordered allocation of positionss
such that $R\neq R^{\prime }$, then $R_{pq}^{0}\neq R_{pq}^{\prime 0}$ and $%
R_{pq}\neq R_{pq}^{\prime }$.
\end{lemma}

\bigskip

Lemmas 4 and 5 below will also be useful.

\begin{lemma}
Let $\left( N,T,f\right) $ be a game ladder, $p,q\in N$ two players such
that $p\succeq q$, and $R=\left( sR,x^{R}\right) $ an ordered allocation of
positions such that $R\in $ $\mathcal{R}_{iq}^{+}\left( f\right) $, $f(x^{%
{\small R}_{{\small q}}^{{\small j}}})\geq z_{i}$ and \newline
$f(x^{{\small R}_{{\small q}}^{{\small j}}}+(j-x_{q}^{{\small R}_{{\small q}%
}^{{\small j}}})e^{q})<z_{i}$.

$1)$ If $sR\left( p\right) <sR\left( q\right) $ and $x_{p}^{R}>x_{q}^{R}$,
then $R_{pq}\in $ $\mathcal{R}_{ip}^{+}\left( f\right) .$

$2)$ If $sR\left( p\right) >sR\left( q\right) $, then $R_{pq}\in $ $\mathcal{%
R}_{ip}^{+}\left( f\right) $.
\end{lemma}

\begin{proof}
$1)$ Suppose that $sR\left( p\right) <sR\left( q\right) $ and $%
x_{p}^{R}=r>s=x_{q}^{R}$. Let us show that $R_{pq}\in $ $\mathcal{R}%
_{ip}^{+}\left( f\right) $, that is:

\begin{equation*}
\left\{ 
\begin{array}{c}
f(x^{\left( {\small R}_{{\small pq}}\right) _{{\small p}}^{{\small j}}})\geq
z_{i} \\ 
f\left( x^{\left( {\small R}_{{\small pq}}\right) _{{\small p}}^{{\small j}%
}}+(j-x_{p}^{\left( {\small R}_{{\small pq}}\right) _{{\small p}}^{{\small j}%
}})e^{p}\right) <z_{i}%
\end{array}%
\right. \text{ or }\left\{ 
\begin{array}{c}
f(x^{\left( {\small R}_{{\small pq}}\right) _{{\small p}}^{{\small 1}%
}})<z_{i} \\ 
f(x^{\left( {\small R}_{{\small pq}}\right) _{{\small p}}^{{\small 1}%
}}+(1-x_{p}^{\left( {\small R}_{{\small pq}}\right) _{{\small p}}^{{\small 1}%
}})e^{p})\geq z_{i}%
\end{array}%
\right.
\end{equation*}%
Since \newline
\begin{equation*}
p\succeq q\text{ and }x_{p}^{{\small R}_{{\small q}}^{{\small j}}}>x_{q}^{%
{\small R}_{{\small q}}^{{\small j}}}
\end{equation*}

it follows that: \newline
\begin{equation*}
f\left( x^{{\small R}_{{\small q}}^{{\small j}}}+(x_{q}^{{\small R}_{{\small %
q}}^{{\small j}}}-x_{p}^{{\small R}_{{\small q}}^{{\small j}}})e^{p}+(x_{p}^{%
{\small R}_{{\small q}}^{{\small j}}}-x_{q}^{{\small R}_{{\small q}}^{%
{\small j}}})e^{q}\right) \geq f(x^{{\small R}_{{\small q}}^{{\small j}%
}})\geq z_{i}\text{.}
\end{equation*}%
But 
\begin{equation*}
x^{{\small R}_{{\small q}}^{{\small j}}}+(x_{q}^{{\small R}_{{\small q}}^{%
{\small j}}}-x_{p}^{{\small R}_{{\small q}}^{{\small j}}})e^{p}+(x_{p}^{%
{\small R}_{{\small q}}^{{\small j}}}-x_{q}^{{\small R}_{{\small q}}^{%
{\small j}}})e^{q}=x^{\left( {\small R}_{{\small pq}}\right) _{{\small p}}^{%
{\small j}}}\text{.}
\end{equation*}

Hence 
\begin{equation*}
f(x^{\left( {\small R}_{{\small pq}}\right) _{{\small p}}^{{\small j}}})\geq
z_{i}\text{.}
\end{equation*}

Let us show that $f\left( x^{\left( {\small R}_{{\small pq}}\right) _{%
{\small p}}^{{\small j}}}+(j-x_{p}^{\left( {\small R}_{{\small pq}}\right) _{%
{\small p}}^{{\small j}}})e^{p}\right) <z_{i}$.

We have: 
\begin{equation*}
f(x^{{\small R}_{{\small q}}^{{\small j}}}+(r-x_{q}^{{\small R}_{{\small q}%
}^{{\small j}}})e^{q})<z_{i}\text{ or }f(x^{{\small R}_{{\small q}}^{{\small %
j}}}+(r-x_{q}^{{\small R}_{{\small q}}^{{\small j}}})e^{q})\geq z_{i}\text{.}
\end{equation*}

\begin{itemize}
\item If $f(x^{{\small R}_{{\small q}}^{{\small j}}}+(r-x_{q}^{{\small R}_{%
{\small q}}^{{\small j}}})e^{q})<z_{i}$, since 
\begin{equation*}
x^{{\small R}_{{\small q}}^{{\small j}}}+(j-x_{p}^{{\small R}_{{\small q}}^{%
{\small j}}})e^{p}+(r-x_{q}^{{\small R}_{{\small q}}^{{\small j}}})e^{q}\leq
x^{{\small R}_{{\small q}}^{{\small j}}}+(r-x_{q}^{{\small R}_{{\small q}}^{%
{\small j}}})e^{q}
\end{equation*}%
and \newline
\begin{equation*}
x^{{\small R}_{{\small q}}^{{\small j}}}+(j-x_{p}^{{\small R}_{{\small q}}^{%
{\small j}}})e^{p}+(r-x_{q}^{{\small R}_{{\small q}}^{{\small j}%
}})e^{q}=x^{\left( {\small R}_{{\small pq}}\right) _{{\small p}}^{{\small j}%
}}+(j-x_{p}^{\left( {\small R}_{{\small pq}}\right) _{{\small p}}^{{\small j}%
}})e^{p}\text{.}
\end{equation*}%
it follows that:%
\begin{equation*}
f(x^{\left( {\small R}_{{\small pq}}\right) _{{\small p}}^{{\small j}%
}}+(j-x_{p}^{\left( {\small R}_{{\small pq}}\right) _{{\small p}}^{{\small j}%
}})e^{p})=f(x^{{\small R}_{{\small q}}^{{\small j}}}+(j-x_{p}^{{\small R}_{%
{\small q}}^{{\small j}}})e^{p}+(r-x_{q}^{{\small R}_{{\small q}}^{{\small j}%
}})e^{q})\leq f(x^{{\small R}_{{\small q}}^{{\small j}}}+(r-x_{q}^{{\small R}%
_{{\small q}}^{{\small j}}})e^{q})<z_{i}\text{.}
\end{equation*}%
Therefore:%
\begin{equation*}
f(x^{\left( {\small R}_{{\small pq}}\right) _{{\small p}}^{{\small j}%
}}+(j-x_{p}^{\left( {\small R}_{{\small pq}}\right) _{{\small p}}^{{\small j}%
}})e^{p})<z_{i}\text{.}
\end{equation*}

\item If $f(x^{{\small R}_{{\small q}}^{{\small j}}}+(r-x_{q}^{{\small R}_{%
{\small q}}^{{\small j}}})e^{q})\geq z_{i}$, since 
\begin{equation*}
x^{{\small R}_{{\small q}}^{{\small j}}}+(j-x_{q}^{{\small R}_{{\small q}}^{%
{\small j}}})e^{q}=x^{{\small R}_{{\small q}}^{{\small j}}}+(r-x_{q}^{%
{\small R}_{{\small q}}^{{\small j}}})e^{q}+(j-r)e^{q}
\end{equation*}%
it follows that 
\begin{equation*}
f(x^{{\small R}_{{\small q}}^{{\small j}}}+(r-x_{q}^{{\small R}_{{\small q}%
}^{{\small j}}})e^{q}+(j-r)e^{q})<z_{i}
\end{equation*}
\end{itemize}

because 
\begin{equation*}
f(x^{{\small R}_{{\small q}}^{{\small j}}}+(j-x_{q}^{{\small R}_{{\small q}%
}^{{\small j}}})e^{q})<z_{i}\text{.}
\end{equation*}

We therefore have: 
\begin{equation*}
f(x^{{\small R}_{{\small q}}^{{\small j}}}+(r-x_{q}^{{\small R}_{{\small q}%
}^{{\small j}}})e^{q})>f(x^{{\small R}_{{\small q}}^{{\small j}}}+(r-x_{q}^{%
{\small R}_{{\small q}}^{{\small j}}})e^{q}+(j-r)e^{q})\text{.}
\end{equation*}%
\ Since $p\succeq q$, we have 
\begin{equation*}
f(x^{{\small R}_{{\small q}}^{{\small j}}}+(r-x_{q}^{{\small R}_{{\small q}%
}^{{\small j}}})e^{q}+(j-x_{p}^{{\small R}_{{\small q}}^{{\small j}%
}})e^{q})<z_{i}\text{.}
\end{equation*}

But 
\begin{equation*}
x^{{\small R}_{{\small q}}^{{\small j}}}+(r-x_{q}^{{\small R}_{{\small q}}^{%
{\small j}}})e^{q}+(j-x_{p}^{{\small R}_{{\small q}}^{{\small j}%
}})e^{q}=x^{\left( {\small R}_{{\small pq}}\right) _{{\small p}}^{{\small j}%
}}+(j-x_{p}^{\left( {\small R}_{{\small pq}}\right) _{{\small p}}^{{\small j}%
}})e^{p}\text{.}
\end{equation*}%
Hence 
\begin{equation*}
f(x^{\left( {\small R}_{{\small pq}}\right) _{{\small p}}^{{\small j}%
}}+(j-x_{p}^{\left( {\small R}_{{\small pq}}\right) _{{\small p}}^{{\small j}%
}})e^{p})<z_{i}\text{.}
\end{equation*}%
\newline
Therefore 
\begin{equation*}
f(x^{\left( {\small R}_{{\small pq}}\right) _{{\small p}}^{{\small j}}})\geq
z_{i}\text{ and }f(x^{\left( {\small R}_{{\small pq}}\right) _{{\small p}}^{%
{\small j}}}+(j-x_{p}^{\left( {\small R}_{{\small pq}}\right) _{{\small p}}^{%
{\small j}}})e^{p})<z_{i}\text{.}
\end{equation*}

We conclude that 
\begin{equation*}
R_{pq}\in \mathcal{R}_{ip}^{+}\left( f\right) \text{.}
\end{equation*}

$2)$ Suppose that $sR\left( p\right) >sR\left( q\right) $ and let us show
that $R_{pq}\in $ $\mathcal{R}_{ip}^{+}\left( f\right) $. Il suffices to
show that:%
\begin{equation*}
\left\{ 
\begin{array}{c}
f(x^{\left( {\small R}_{{\small pq}}\right) _{{\small p}}^{{\small j}}})\geq
z_{i} \\ 
f(x^{\left( {\small R}_{{\small pq}}\right) _{{\small p}}^{{\small j}%
}}+(j-x_{p}^{\left( {\small R}_{{\small pq}}\right) _{{\small p}}^{{\small j}%
}})e^{p})<z_{i}%
\end{array}%
\right.
\end{equation*}%
$.$

By assumption, we have: 
\begin{equation*}
R\in \mathcal{R}_{iq}^{+}\left( f\right) ,f(x^{{\small R}_{{\small q}}^{%
{\small j}}})\geq z_{i}\text{ and }f(x^{{\small R}_{{\small q}}^{{\small j}%
}}+(j-x_{q}^{{\small R}_{{\small q}}^{{\small j}}})e^{q})<z_{i}\text{.}
\end{equation*}

But 
\begin{equation*}
x^{{\small R}_{{\small q}}^{{\small j}}}+(j-x_{q}^{{\small R}_{{\small q}}^{%
{\small j}}})e^{q}\leq x^{{\small R}_{{\small q}}^{{\small j}}}+(j-x_{q}^{%
{\small R}_{{\small q}}^{{\small j}}})e^{q}+(s-j)e^{q}=x^{{\small R}_{%
{\small q}}^{{\small j}}}
\end{equation*}%
because 
\begin{equation*}
x_{q}^{R}=s\text{.}
\end{equation*}

Hence $\ $%
\begin{equation*}
f(x^{{\small R}_{{\small q}}^{{\small j}}}+(j-x_{q}^{{\small R}_{{\small q}%
}^{{\small j}}})e^{q}+(s-j)e^{q})=f(x^{{\small R}_{{\small q}}^{{\small j}%
}})\geq z_{i}>f(x^{{\small R}_{{\small q}}^{{\small j}}}+(j-x_{q}^{{\small R}%
_{{\small q}}^{{\small j}}})e^{q})\text{.}
\end{equation*}%
Since $p\succeq q$, we have: $\ $%
\begin{equation*}
f(x^{{\small R}_{{\small q}}^{{\small j}}}+(j-x_{q}^{{\small R}_{{\small q}%
}^{{\small j}}})e^{q}+(s-x_{p}^{{\small R}_{{\small q}}^{{\small j}%
}})e^{q})\geq f(x^{{\small R}_{{\small q}}^{{\small j}}}+(j-x_{q}^{{\small R}%
_{{\small q}}^{{\small j}}})e^{q}+(s-j)e^{q})\geq z_{i}\text{.}
\end{equation*}

But $\ $%
\begin{equation*}
x^{{\small R}_{{\small q}}^{{\small j}}}+(j-x_{q}^{{\small R}_{{\small q}}^{%
{\small j}}})e^{q}+(s-x_{p}^{{\small R}_{{\small q}}^{{\small j}%
}})e^{q}=x^{\left( {\small R}_{{\small pq}}\right) _{{\small p}}^{{\small j}%
}}\text{.}
\end{equation*}

Hence%
\begin{equation*}
f(x^{\left( {\small R}_{{\small pq}}\right) _{{\small p}}^{{\small j}}})\geq
z_{i}\text{.}
\end{equation*}

We have: 
\begin{equation*}
x^{\left( {\small R}_{{\small pq}}\right) _{{\small p}}^{{\small j}%
}}+(j-x_{p}^{\left( {\small R}_{{\small pq}}\right) _{{\small p}}^{{\small j}%
}})e^{p}=x^{{\small R}_{{\small q}}^{{\small j}}}+(j-x_{q}^{{\small R}_{%
{\small q}}^{{\small j}}})e^{q}\text{.}
\end{equation*}%
Hence 
\begin{equation*}
f(x^{\left( {\small R}_{{\small pq}}\right) _{{\small p}}^{{\small j}%
}}+(j-x_{p}^{\left( {\small R}_{{\small pq}}\right) _{{\small p}}^{{\small j}%
}})e^{p})<z_{i}
\end{equation*}%
as 
\begin{equation*}
f(x^{{\small R}_{{\small q}}^{{\small j}}}+(j-x_{q}^{{\small R}_{{\small q}%
}^{{\small j}}})e^{q})<z_{i}\text{.}
\end{equation*}

We therefore have \newline
\begin{equation*}
f(x^{\left( {\small R}_{{\small pq}}\right) _{{\small p}}^{{\small j}}})\geq
z_{i}\text{ and }f(x^{\left( {\small R}_{{\small pq}}\right) _{{\small p}}^{%
{\small j}}}+(j-x_{p}^{\left( {\small R}_{{\small pq}}\right) _{{\small p}}^{%
{\small j}}})e^{p})<z_{i}
\end{equation*}%
which implies: \newline
\begin{equation*}
R_{pq}\in \mathcal{R}_{ip}^{+}\left( f\right) \text{.}
\end{equation*}
\end{proof}

\bigskip

We also have the following lemma.

\begin{lemma}
Let $\left( N,T,f\right) $ be a game ladder$,$ $p,$ $q\in N$ two players
such that $p\succeq q$, and $R=\left( sR,\text{ }x^{R}\right) $ an ordered
allocation of positions such that $R\in $ $\mathcal{R}_{iq}^{+}(f)$, $%
f\left( x^{{\small R}_{{\small q}}^{{\small 1}}}\right) <z_{i}$ and \newline
$f\left( x^{{\small R}_{{\small q}}^{{\small 1}}}+(1-x_{q}^{{\small R}_{%
{\small q}}^{{\small 1}}})e^{q}\right) \geq z_{i}$.

$1)$ If $sR\left( p\right) <sR\left( q\right) $ and $x_{p}^{R}<x_{q}^{R}$,
then $R_{pq}\in $ $\mathcal{R}_{ip}^{+}\left( f\right) $.

$2)$ If $sR\left( p\right) >sR\left( q\right) $, then $R_{pq}\in $ $\mathcal{%
R}_{ip}^{+}\left( f\right) $.
\end{lemma}

\begin{proof}
$1)$ Suppose that $sR\left( p\right) <sR\left( q\right) $ and $%
x_{p}^{R}=r>s=x_{q}^{R}$. Let us show that $R_{pq}\in $ $\mathcal{R}%
_{ip}^{+}\left( f\right) $, that is:

\begin{equation*}
\left\{ 
\begin{array}{c}
f(x^{\left( {\small R}_{{\small pq}}\right) _{{\small p}}^{{\small j}}})\geq
z_{i} \\ 
f(x^{\left( {\small R}_{{\small pq}}\right) _{{\small p}}^{{\small j}%
}}+(j-x_{p}^{\left( {\small R}_{{\small pq}}\right) _{{\small p}}^{{\small j}%
}})e^{p})<z_{i}%
\end{array}%
\right. \text{ or }\left\{ 
\begin{array}{c}
f(x^{\left( {\small R}_{{\small pq}}\right) _{{\small p}}^{{\small 1}%
}})<z_{i} \\ 
f(x^{\left( {\small R}_{{\small pq}}\right) _{{\small p}}^{{\small 1}%
}}+(1-x_{p}^{\left( {\small R}_{{\small pq}}\right) _{{\small p}}^{{\small 1}%
}})e^{p})\geq z_{i}%
\end{array}%
\right.
\end{equation*}%
Since $p\succeq q$, it follows that:

\begin{equation*}
f(x^{{\small R}_{{\small q}}^{{\small 1}}}+(x_{p}^{{\small R}_{{\small q}}^{%
{\small 1}}}-x_{q}^{{\small R}_{{\small q}}^{{\small 1}}})e^{q}+(x_{q}^{%
{\small R}_{{\small q}}^{{\small 1}}}-x_{p}^{{\small R}_{{\small q}}^{%
{\small j}}})e^{p})\leq f\left( x^{{\small R}_{{\small q}}^{{\small 1}%
}}\right) <z_{i}\text{.}
\end{equation*}

But:

\begin{equation*}
x^{{\small R}_{{\small q}}^{{\small 1}}}+(x_{p}^{{\small R}_{{\small q}}^{%
{\small 1}}}-x_{q}^{{\small R}_{{\small q}}^{{\small 1}}})e^{q}+(x_{q}^{%
{\small R}_{{\small q}}^{{\small 1}}}-x_{p}^{{\small R}_{{\small q}}^{%
{\small j}}})e^{p}=x^{\left( {\small R}_{{\small pq}}\right) _{{\small p}}^{%
{\small 1}}}\text{.}
\end{equation*}

So

\begin{equation*}
f(x^{\left( {\small R}_{{\small pq}}\right) _{{\small p}}^{{\small 1}%
}})<z_{i}\text{.}
\end{equation*}

Let us show that $f(x^{\left( {\small R}_{{\small pq}}\right) _{{\small p}}^{%
{\small 1}}}+(1-x_{p}^{\left( {\small R}_{{\small pq}}\right) _{{\small p}}^{%
{\small 1}}})e^{p})\geq z_{i}$. We have

\begin{equation*}
f(x^{{\small R}_{{\small q}}^{{\small 1}}}+(r-x_{q}^{{\small R}_{{\small q}%
}^{{\small 1}}})e^{q})\geq z_{i}\text{ or }f(x^{{\small R}_{{\small q}}^{%
{\small 1}}}+(r-x_{q}^{{\small R}_{{\small q}}^{{\small 1}}})e^{q})<z_{i}%
\text{.}
\end{equation*}

\begin{itemize}
\item Suppose $f(x^{{\small R}_{{\small q}}^{{\small 1}}}+(r-x_{q}^{{\small R%
}_{{\small q}}^{{\small 1}}})e^{q})\geq z_{i}$. Since 
\begin{equation*}
x^{{\small R}_{{\small q}}^{{\small 1}}}+(r-x_{q}^{{\small R}_{{\small q}}^{%
{\small 1}}})e^{q}\leq x^{{\small R}_{{\small q}}^{{\small 1}}}+(r-x_{q}^{%
{\small R}_{{\small q}}^{{\small 1}}})e^{q}+(1-x_{p}^{{\small R}_{{\small q}%
}^{{\small 1}}})e^{q}=x^{\left( {\small R}_{{\small pq}}\right) _{{\small p}%
}^{{\small 1}}}+(1-x_{p}^{\left( {\small R}_{{\small pq}}\right) _{{\small p}%
}^{{\small 1}}})e^{p}
\end{equation*}%
or equivalently 
\begin{equation*}
z_{i}\leq f(x^{{\small R}_{{\small q}}^{{\small 1}}}+(r-x_{q}^{{\small R}_{%
{\small q}}^{{\small 1}}})e^{q})\leq f(x^{{\small R}_{{\small q}}^{{\small 1}%
}}+(r-x_{q}^{{\small R}_{{\small q}}^{{\small 1}}})e^{q}+(1-x_{p}^{{\small R}%
_{{\small q}}^{{\small 1}}})e^{q})=f(x^{\left( {\small R}_{{\small pq}%
}\right) _{{\small p}}^{{\small 1}}}+(1-x_{p}^{\left( {\small R}_{{\small pq}%
}\right) _{{\small p}}^{{\small 1}}})e^{p})
\end{equation*}%
it follows that: 
\begin{equation*}
f(x^{\left( {\small R}_{{\small pq}}\right) _{{\small p}}^{{\small 1}%
}}+(1-x_{p}^{\left( {\small R}_{{\small pq}}\right) _{{\small p}}^{{\small 1}%
}})e^{p})\geq z_{i}\text{.}
\end{equation*}

\item If $f(x^{{\small R}_{{\small q}}^{{\small 1}}}+(r-x_{q}^{{\small R}_{%
{\small q}}^{{\small 1}}})e^{q})<z_{i}$, since 
\begin{equation*}
x^{{\small R}_{{\small q}}^{{\small 1}}}+(r-x_{q}^{{\small R}_{{\small q}}^{%
{\small 1}}})e^{q}+(1-r)e^{q}=x^{{\small R}_{{\small q}}^{{\small 1}%
}}+(1-x_{q}^{{\small R}_{{\small q}}^{{\small 1}}})e^{q}
\end{equation*}
\end{itemize}

we have%
\begin{equation*}
f(x^{{\small R}_{{\small q}}^{{\small 1}}}+(r-x_{q}^{{\small R}_{{\small q}%
}^{{\small 1}}})e^{q}+(1-r)e^{q})=f(x^{{\small R}_{{\small q}}^{{\small 1}%
}}+(1-x_{q}^{{\small R}_{{\small q}}^{{\small 1}}})e^{q})\geq z_{i}>f(x^{%
{\small R}_{{\small q}}^{{\small 1}}}+(r-x_{q}^{{\small R}_{{\small q}}^{%
{\small 1}}})e^{q})\text{.}
\end{equation*}%
That is 
\begin{equation*}
f(x^{{\small R}_{{\small q}}^{{\small 1}}}+(r-x_{q}^{{\small R}_{{\small q}%
}^{{\small 1}}})e^{q}+(1-r)e^{q})>f(x^{{\small R}_{{\small q}}^{{\small 1}%
}}+(r-x_{q}^{{\small R}_{{\small q}}^{{\small 1}}})e^{q})\text{.}
\end{equation*}

Since $p\succeq q$, we have 
\begin{equation*}
f(x^{{\small R}_{{\small q}}^{{\small 1}}}+(r-x_{q}^{{\small R}_{{\small q}%
}^{{\small 1}}})e^{q}+(1-x_{p}^{{\small R}_{{\small q}}^{{\small 1}%
}})e^{p})\geq f(x^{{\small R}_{{\small q}}^{{\small 1}}}+(r-x_{q}^{{\small R}%
_{{\small q}}^{{\small 1}}})e^{q}+(1-r)e^{q})\geq z_{i}\text{.}
\end{equation*}%
But 
\begin{equation*}
x^{{\small R}_{{\small q}}^{{\small 1}}}+(r-x_{q}^{{\small R}_{{\small q}}^{%
{\small 1}}})e^{q}+(t_{1}-x_{p}^{{\small R}_{{\small q}}^{{\small 1}%
}})e^{p}=x^{\left( {\small R}_{{\small pq}}\right) _{{\small p}}^{{\small 1}%
}}+(t_{1}-x_{p}^{\left( {\small R}_{{\small pq}}\right) _{{\small p}}^{%
{\small 1}}})e^{p}\text{.}
\end{equation*}

Hence 
\begin{equation*}
f(x^{\left( {\small R}_{{\small pq}}\right) _{{\small p}}^{{\small 1}%
}}+(t_{1}-x_{p}^{\left( {\small R}_{{\small pq}}\right) _{{\small p}}^{%
{\small 1}}})e^{p})\geq z_{i}\text{.}
\end{equation*}%
\newline

We then have 
\begin{equation*}
f(x^{\left( {\small R}_{{\small pq}}\right) _{{\small p}}^{{\small 1}%
}})<z_{i}\text{ and }f(x^{\left( {\small R}_{{\small pq}}\right) _{{\small p}%
}^{{\small 1}}}+(1-x_{p}^{\left( {\small R}_{{\small pq}}\right) _{{\small p}%
}^{{\small 1}}})e^{p})\geq z_{i}
\end{equation*}%
\newline
which implies 
\begin{equation*}
R_{pq}\in \mathcal{R}_{ip}^{+}\left( f\right) \text{.}
\end{equation*}

$2)$ Suppose $sR\left( p\right) >sR\left( q\right) $ and let us show that $%
R_{pq}\in $ $\mathcal{R}_{ip}^{+}\left( f\right) $, that is:%
\begin{equation*}
\left\{ 
\begin{array}{c}
f(x^{\left( {\small R}_{{\small pq}}\right) _{{\small p}}^{{\small j}}})\geq
z_{i} \\ 
f(x^{\left( {\small R}_{{\small pq}}\right) _{{\small p}}^{{\small j}%
}}+(j-x_{p}^{\left( {\small R}_{{\small pq}}\right) _{{\small p}}^{{\small j}%
}})e^{p})<z_{i}%
\end{array}%
\right. \text{ or }\left\{ 
\begin{array}{c}
f(x^{\left( {\small R}_{{\small pq}}\right) _{{\small p}}^{{\small 1}%
}})<z_{i} \\ 
f(x^{\left( {\small R}_{{\small pq}}\right) _{{\small p}}^{{\small 1}%
}}+(1-x_{p}^{\left( {\small R}_{{\small pq}}\right) _{{\small p}}^{{\small 1}%
}})e^{p})\geq z_{i}%
\end{array}%
\right.
\end{equation*}%
By assumption, we have 
\begin{equation*}
R\in \mathcal{R}_{iq}^{+}\left( f\right) ,f\left( x^{{\small R}_{{\small q}%
}^{{\small 1}}}\right) <z_{i}\text{ and }f(x^{{\small R}_{{\small q}}^{%
{\small 1}}}+(1-x_{q}^{{\small R}_{{\small q}}^{{\small 1}}})e^{q})\geq z_{i}%
\text{.}
\end{equation*}

But 
\begin{equation*}
x^{{\small R}_{{\small q}}^{{\small 1}}}=x^{{\small R}_{{\small q}}^{{\small %
1}}}+(1-x_{q}^{{\small R}_{{\small q}}^{{\small 1}}})e^{q}+(s-1)e^{q}\leq x^{%
{\small R}_{{\small q}}^{{\small 1}}}+(1-x_{q}^{{\small R}_{{\small q}}^{%
{\small 1}}})e^{q}\text{.}
\end{equation*}

So%
\begin{equation*}
f\left( x^{{\small R}_{{\small q}}^{{\small 1}}}\right) =f(x^{{\small R}_{%
{\small q}}^{{\small 1}}}+(1-x_{q}^{{\small R}_{{\small q}}^{{\small 1}%
}})e^{q}+(s-1)e^{q})<z_{i}\leq f(x^{{\small R}_{{\small q}}^{{\small 1}%
}}+(1-x_{q}^{{\small R}_{{\small q}}^{{\small 1}}})e^{q})\text{.}
\end{equation*}

which implies 
\begin{equation*}
f(x^{{\small R}_{{\small q}}^{{\small 1}}}+(1-x_{q}^{{\small R}_{{\small q}%
}^{{\small 1}}})e^{q})>f(x^{{\small R}_{{\small q}}^{{\small 1}}}+(1-x_{q}^{%
{\small R}_{{\small q}}^{{\small 1}}})e^{q}+(s-1)e^{q})\text{.}
\end{equation*}

Given that $p\succeq q$, we have 
\begin{equation*}
f(x^{{\small R}_{{\small q}}^{{\small 1}}}+(1-x_{q}^{{\small R}_{{\small q}%
}^{{\small 1}}})e^{q}+(s-x_{p}^{{\small R}_{{\small q}}^{{\small 1}%
}})e^{p})\leq f(x^{{\small R}_{{\small q}}^{{\small 1}}}+(1-x_{q}^{{\small R}%
_{{\small q}}^{{\small 1}}})e^{q}+(s-1)e^{q})<z_{i}\text{.}
\end{equation*}%
But 
\begin{equation*}
f(x^{{\small R}_{{\small q}}^{{\small 1}}}+(1-x_{q}^{{\small R}_{{\small q}%
}^{{\small 1}}})e^{q}+(s-x_{p}^{{\small R}_{{\small q}}^{{\small 1}%
}})e^{p})=x^{\left( {\small R}_{{\small pq}}\right) _{{\small p}}^{{\small 1}%
}}\text{.}
\end{equation*}

Hence 
\begin{equation*}
f(x^{\left( {\small R}_{{\small pq}}\right) _{{\small p}}^{{\small 1}%
}})<z_{i}\text{.}
\end{equation*}

We have 
\begin{equation*}
x^{\left( {\small R}_{{\small pq}}\right) _{{\small p}}^{{\small 1}%
}}+(1-x_{p}^{\left( {\small R}_{{\small pq}}\right) _{{\small p}}^{{\small 1}%
}})e^{p}=x^{{\small R}_{{\small q}}^{{\small 1}}}+(1-x_{q}^{{\small R}_{%
{\small q}}^{{\small 1}}})e^{q}
\end{equation*}

which implies 
\begin{equation*}
f(x^{\left( {\small R}_{{\small pq}}\right) _{{\small p}}^{{\small 1}%
}}+(1-x_{p}^{\left( {\small R}_{{\small pq}}\right) _{{\small p}}^{{\small 1}%
}})e^{p})\geq z_{i}
\end{equation*}

because 
\begin{equation*}
f\left( x^{{\small R}_{{\small q}}^{{\small 1}}}+(1-x_{q}^{{\small R}_{%
{\small q}}^{{\small 1}}})e^{q}\right) \geq z_{i}\text{.}
\end{equation*}

We then obtain that: 
\begin{equation*}
f(x^{\left( {\small R}_{{\small pq}}\right) _{{\small p}}^{{\small 1}%
}})<z_{i}\text{ and }f(x^{\left( {\small R}_{{\small pq}}\right) _{{\small p}%
}^{{\small 1}}}+(1-x_{p}^{\left( {\small R}_{{\small pq}}\right) _{{\small p}%
}^{{\small 1}}})e^{p})\geq z_{i}
\end{equation*}

from which it follows that 
\begin{equation*}
R_{pq}\in \mathcal{R}_{ip}^{+}\left( f\right) \text{.}
\end{equation*}
\end{proof}

\bigskip

In the sequel, given $a\in N$, $i\in \left\{ 1,...,k-1\right\} $, we define
the following sets:%
\begin{equation*}
\begin{array}{c}
\mathcal{R}_{ia}^{+}=\left\{ R=(sR,\text{ }x^{R})\in \text{ }\mathcal{R}%
_{ia}^{+}:\left( 
\begin{tabular}{l}
$f(x^{{\small R}_{{\small a}}^{{\small j}}})\geq z_{i}$ and \\ 
$f(x^{{\small R}_{{\small a}}^{{\small j}}}+(j-x_{a}^{{\small R}_{{\small a}%
}^{{\small 1}}})e^{a})<z_{i}$%
\end{tabular}%
\right) \right\}  \\ 
\mathcal{R}_{ia}^{-}=\left\{ R=(sR,\text{ }x^{R})\in \text{ }\mathcal{R}%
_{ia}^{+}:\left( 
\begin{tabular}{l}
$f\left( x^{{\small R}_{{\small a}}^{{\small 1}}}\right) <z_{i}$ and \\ 
$f(x^{{\small R}_{{\small a}}^{{\small j}}}+(1-x_{a}^{{\small R}_{{\small a}%
}^{{\small 1}}})e^{a})\geq z_{i}$%
\end{tabular}%
\right) \right\} 
\end{array}%
\end{equation*}

\bigskip

The following lemma is needed.

\begin{lemma}
Let $\left( N,T,f\right) $ be a game ladder, $p,$ $q\in N$ two players such
that $p\succeq q$, and $R$ an ordered allocation of positions such that $%
R\in $ $\mathcal{R}_{iq}^{+},$ $sR\left( p\right) <sR\left( q\right) .$

$1)$ If 
\begin{equation*}
\left\{ 
\begin{array}{l}
R\in \text{ }\mathcal{R}_{iq}^{+}\text{ and } \\ 
x_{p}^{R}\leq x_{q}^{R}%
\end{array}%
\right. \text{ or }\left\{ 
\begin{array}{l}
R\in \text{ }\mathcal{R}_{iq}^{+},\text{ }x_{p}^{R}>x_{q}^{R}\text{ and } \\ 
R^{\prime }=\left(
sR,x^{R}+(x_{p}^{R}-x_{q}^{R})e^{q}+(x_{q}^{R}-x_{p}^{R})e^{p}\right) \in 
\text{ }\mathcal{R}_{iq}^{+}%
\end{array}%
\right.
\end{equation*}%
then 
\begin{equation*}
R_{pq}^{0}\in \mathcal{R}_{ip}^{+}\text{.}
\end{equation*}

$2)$ If%
\begin{equation*}
\left\{ 
\begin{array}{l}
R\in \mathcal{R}_{iq}^{-}\text{ and } \\ 
x_{p}^{R}\geq x_{q}^{R}%
\end{array}%
\right. \text{ or }\left\{ 
\begin{array}{l}
R\in \mathcal{R}_{iq}^{-},\text{ }x_{p}^{R}<x_{q}^{R}\text{ and } \\ 
R^{\prime }=\left(
sR,x^{R}+(x_{p}^{R}-x_{q}^{R})e^{q}+(x_{q}^{R}-x_{p}^{R})e^{p}\right) \in 
\mathcal{R}_{iq}^{-}%
\end{array}%
\right.
\end{equation*}%
then 
\begin{equation*}
R_{pq}^{0}\in \mathcal{R}_{ip}^{-}\text{.}
\end{equation*}
\end{lemma}

\begin{proof}
Let $\left( N,T,f\right) $ be a game ladder, $p,$ $q\in N$ and $R$ an
ordered allocation of positions such that $sR\left( p\right) <sR\left(
q\right) $ and $R\in $ $\mathcal{R}_{iq}^{+}$.

$1)$ Suppose $\left\{ 
\begin{array}{l}
R\in \text{ }\mathcal{R}_{iq}^{+} \\ 
x_{p}^{R}=r\leq s=x_{q}^{R}%
\end{array}%
\right. $ and let us show that $R_{pq}^{0}\in $ $\mathcal{R}_{ip}^{+}.$

It suffices to show that: 
\begin{equation*}
\left\{ 
\begin{array}{l}
f(x^{\left( {\small R}_{{\small pq}}^{{\small 0}}\right) _{{\small p}}^{%
{\small j}}})\geq z_{i} \\ 
f(x^{\left( {\small R}_{{\small pq}}^{{\small 0}}\right) _{{\small p}}^{%
{\small j}}}+(j-x_{p}^{\left( {\small R}_{{\small pq}}^{{\small 0}}\right) _{%
{\small p}}^{{\small j}}})e^{p})<z_{i}%
\end{array}%
\right.
\end{equation*}

Since 
\begin{equation*}
x^{{\small R}_{{\small q}}^{{\small j}}}=x^{\left( {\small R}_{{\small pq}}^{%
{\small 0}}\right) _{{\small p}}^{{\small j}}}
\end{equation*}

it follows that 
\begin{equation*}
f(x^{\left( {\small R}_{{\small pq}}^{{\small 0}}\right) _{{\small p}}^{%
{\small j}}})\geq z_{i}
\end{equation*}

because 
\begin{equation*}
f(x^{{\small R}_{{\small q}}^{{\small j}}})\geq z_{i}\text{.}
\end{equation*}

Let us show that 
\begin{equation*}
f(x^{\left( {\small R}_{{\small pq}}^{{\small 0}}\right) _{{\small p}}^{%
{\small j}}}+(j-x_{p}^{\left( {\small R}_{{\small pq}}^{{\small 0}}\right) _{%
{\small p}}^{{\small j}}})e^{p})<z_{i}\text{.}
\end{equation*}

We have 
\begin{equation*}
f(x^{{\small R}_{{\small q}}^{{\small j}}}+(s-x_{p}^{{\small R}_{{\small q}%
}^{{\small j}}})e^{p})\geq z_{i}\text{ or }f(x^{{\small R}_{{\small q}}^{%
{\small j}}}+(s-x_{p}^{{\small R}_{{\small q}}^{{\small j}}})e^{p})<z_{i}
\end{equation*}

with $p$ and $q$ being assigned to the same task $s$ in the task profile $x^{%
{\small R}_{{\small q}}^{{\small j}}}+(s-x_{p}^{{\small R}_{{\small q}}^{%
{\small j}}})e^{p}$.

\begin{itemize}
\item If $f(x^{{\small R}_{{\small q}}^{{\small j}}}+(s-x_{p}^{{\small R}_{%
{\small q}}^{{\small j}}})e^{p})\geq z_{i}$, then, since 
\begin{equation*}
x^{{\small R}_{{\small q}}^{{\small j}}}+(s-x_{p}^{{\small R}_{{\small q}}^{%
{\small j}}})e^{p}+(j-x_{q}^{{\small R}_{{\small q}}^{{\small j}}})e^{q}\leq
x^{{\small R}_{{\small q}}^{{\small j}}}+(j-x_{q}^{{\small R}_{{\small q}}^{%
{\small j}}})e^{q}
\end{equation*}
\end{itemize}

by monotonicity, we have 
\begin{equation*}
f(x^{{\small R}_{{\small q}}^{{\small j}}}+(s-x_{p}^{{\small R}_{{\small q}%
}^{{\small j}}})e^{p}+(j-x_{q}^{{\small R}_{{\small q}}^{{\small j}%
}})e^{q})\leq f(x^{{\small R}_{{\small q}}^{{\small j}}}+(j-x_{q}^{{\small R}%
_{{\small q}}^{{\small j}}})e^{q})<z_{i}
\end{equation*}

which implies: 
\begin{equation*}
f(x^{{\small R}_{{\small q}}^{{\small j}}}+(s-x_{p}^{{\small R}_{{\small q}%
}^{{\small j}}})e^{p}+(j-x_{q}^{{\small R}_{{\small q}}^{{\small j}%
}})e^{q})<f(x^{{\small R}_{{\small q}}^{{\small j}}}+(s-x_{p}^{{\small R}_{%
{\small q}}^{{\small j}}})e^{p}).
\end{equation*}%
Given that 
\begin{equation*}
p\succeq q
\end{equation*}

and 
\begin{equation*}
f(x^{{\small R}_{{\small q}}^{{\small j}}}+(s-x_{p}^{{\small R}_{{\small q}%
}^{{\small j}}})e^{p}+(j-x_{q}^{{\small R}_{{\small q}}^{{\small j}%
}})e^{q})<f(x^{{\small R}_{{\small q}}^{{\small j}}}+(s-x_{p}^{{\small R}_{%
{\small q}}^{{\small j}}})e^{p})
\end{equation*}

we have%
\begin{equation*}
f(x^{{\small R}_{{\small q}}^{{\small j}}}+(s-x_{p}^{{\small R}_{{\small q}%
}^{{\small j}}})e^{p}+(j-s)e^{p})\leq f(x^{{\small R}_{{\small q}}^{{\small j%
}}}+(s-x_{p}^{{\small R}_{{\small q}}^{{\small j}}})e^{p}+(j-x_{q}^{{\small R%
}_{{\small q}}^{{\small j}}})e^{q})<z_{i}\text{.}
\end{equation*}%
But 
\begin{equation*}
x^{{\small R}_{{\small q}}^{{\small j}}}+(s-x_{p}^{{\small R}_{{\small q}}^{%
{\small j}}})e^{p}+(j-s)e^{p}=x^{{\small R}_{{\small q}}^{{\small j}%
}}+(j-x_{p}^{{\small R}_{{\small q}}^{{\small j}}})e^{p}=x^{\left( {\small R}%
_{{\small pq}}^{{\small 0}}\right) _{{\small p}}^{{\small j}%
}}+(j-x_{p}^{\left( {\small R}_{{\small pq}}^{{\small 0}}\right) _{{\small p}%
}^{{\small j}}})e^{p}
\end{equation*}

that is, 
\begin{equation*}
f(x^{\left( {\small R}_{{\small pq}}^{{\small 0}}\right) _{{\small p}}^{%
{\small j}}}+(j-x_{p}^{\left( {\small R}_{{\small pq}}^{{\small 0}}\right) _{%
{\small p}}^{{\small j}}})e^{p})<z_{i}\text{.}
\end{equation*}

\begin{itemize}
\item If $f(x^{{\small R}_{{\small q}}^{{\small j}}}+(s-x_{p}^{{\small R}_{%
{\small q}}^{{\small j}}})e^{p})<z_{i}$, then, since
\end{itemize}

\begin{equation*}
x^{{\small R}_{{\small q}}^{{\small j}}}+(j-x_{p}^{{\small R}_{{\small q}}^{%
{\small j}}})e^{p}\leq x^{{\small R}_{{\small q}}^{{\small j}}}+(s-x_{p}^{%
{\small R}_{{\small q}}^{{\small j}}})e^{p}
\end{equation*}

by monotonicity, we have 
\begin{equation*}
f(x^{{\small R}_{{\small q}}^{{\small j}}}+(j-x_{p}^{{\small R}_{{\small q}%
}^{{\small j}}})e^{p})\leq f(x^{{\small R}_{{\small q}}^{{\small j}%
}}+(s-x_{p}^{{\small R}_{{\small q}}^{{\small j}}})e^{p})<z_{i}
\end{equation*}%
that is 
\begin{equation*}
f(x^{\left( {\small R}_{{\small pq}}^{{\small 0}}\right) _{{\small p}}^{%
{\small j}}}+(j-x_{p}^{\left( {\small R}_{{\small pq}}^{{\small 0}}\right) _{%
{\small p}}^{{\small j}}})e^{p})<z_{i}
\end{equation*}

because 
\begin{equation*}
x^{{\small R}_{{\small q}}^{{\small j}}}=x^{\left( {\small R}_{{\small pq}}^{%
{\small 0}}\right) _{{\small p}}^{{\small j}}}\text{.}
\end{equation*}

We then have 
\begin{equation*}
f(x^{\left( {\small R}_{{\small pq}}^{{\small 0}}\right) _{{\small p}}^{%
{\small j}}}+(j-x_{p}^{\left( {\small R}_{{\small pq}}^{{\small 0}}\right) _{%
{\small p}}^{{\small j}}})e^{p})<z_{i}\text{.}
\end{equation*}%
\newline
Therefore 
\begin{equation*}
f(x^{\left( {\small R}_{{\small pq}}^{{\small 0}}\right) _{{\small p}}^{%
{\small j}}})\geq z_{i}\text{ and }f(x^{\left( {\small R}_{{\small pq}}^{%
{\small 0}}\right) _{{\small p}}^{{\small j}}}+(j-x_{p}^{\left( {\small R}_{%
{\small pq}}^{{\small 0}}\right) _{{\small p}}^{{\small j}}})e^{p})<z_{i}%
\text{.}
\end{equation*}

Hence 
\begin{equation*}
R_{pq}^{0}\in \mathcal{R}_{ip}^{+}\text{.}
\end{equation*}

Suppose that $\left\{ 
\begin{array}{l}
R\in \text{ }\mathcal{R}_{iq}^{+},\text{ }x_{p}^{R}=s>r=x_{q}^{R}\text{ and }
\\ 
R^{\prime }=\left(
sR,x_{p}^{R}+(x_{p}^{R}-x_{q}^{R})e^{q}+(x_{q}^{R}-x_{p}^{R})e^{p}\right)
\in \text{ }\mathcal{R}_{iq}^{+}%
\end{array}%
\right. $ and let us show that $R_{pq}^{0}\in $ $\mathcal{R}_{ip}^{+}$.

Since 
\begin{equation*}
x^{\left( {\small R}_{{\small pq}}^{{\small 0}}\right) _{{\small p}}^{%
{\small j}}}=x^{{\small R}_{{\small q}}^{{\small j}}}
\end{equation*}

we have 
\begin{equation*}
f(x^{\left( {\small R}_{{\small pq}}^{{\small 0}}\right) _{{\small p}}^{%
{\small j}}})\geq z_{i}
\end{equation*}

because 
\begin{equation*}
f(x^{{\small R}_{{\small q}}^{{\small j}}})\geq z_{i}
\end{equation*}

Let us show that $f(x^{\left( {\small R}_{{\small pq}}^{{\small 0}}\right) _{%
{\small p}}^{{\small j}}}+(j-x_{p}^{\left( {\small R}_{{\small pq}}^{{\small %
0}}\right) _{{\small p}}^{{\small j}}})e^{p})<z_{i}$.

By monotonicity, we have 
\begin{equation*}
f(x^{\left( {\small R}_{{\small pq}}^{{\small 0}}\right) _{{\small p}}^{%
{\small j}}}+(r-x_{p}^{\left( {\small R}_{{\small pq}}^{{\small 0}}\right) _{%
{\small p}}^{{\small j}}})e^{p})\geq z_{i}\text{.}
\end{equation*}

Since 
\begin{equation*}
x^{{\small R}_{{\small q}}^{\prime {\small j}}}+(j-x_{q}^{{\small R}_{%
{\small q}}^{\prime {\small j}}})e^{q}=x^{\left( {\small R}_{{\small pq}}^{%
{\small 0}}\right) _{{\small p}}^{{\small j}}}+(r-x_{p}^{\left( {\small R}_{%
{\small pq}}^{{\small 0}}\right) _{{\small p}}^{{\small j}%
}})e^{p}+(j-x_{q}^{\left( {\small R}_{{\small pq}}^{{\small 0}}\right) _{%
{\small p}}^{{\small j}}})e^{q}
\end{equation*}

and\ 
\begin{equation*}
f(x^{{\small R}_{{\small q}}^{\prime {\small j}}}+(j-x_{q}^{{\small R}_{%
{\small q}}^{\prime {\small j}}})e^{q})<z_{i}
\end{equation*}%
\newline
it follows that 
\begin{equation*}
f(x^{\left( {\small R}_{{\small pq}}^{{\small 0}}\right) _{{\small p}}^{%
{\small j}}}+(r-x_{p}^{\left( {\small R}_{{\small pq}}^{{\small 0}}\right) _{%
{\small p}}^{{\small j}}})e^{p}+(j-x_{q}^{\left( {\small R}_{{\small pq}}^{%
{\small 0}}\right) _{{\small p}}^{{\small j}}})e^{q})<z_{i}
\end{equation*}%
or equivalently 
\begin{equation*}
f(x^{\left( {\small R}_{{\small pq}}^{{\small 0}}\right) _{{\small p}}^{%
{\small j}}}+(r-x_{p}^{\left( {\small R}_{{\small pq}}^{{\small 0}}\right) _{%
{\small p}}^{{\small j}}})e^{p})>f(x^{\left( {\small R}_{{\small pq}}^{%
{\small 0}}\right) _{{\small p}}^{{\small j}}}+(r-x_{p}^{\left( {\small R}_{%
{\small pq}}^{{\small 0}}\right) _{{\small p}}^{{\small j}%
}})e^{p}+(j-x_{q}^{\left( {\small R}_{{\small pq}}^{{\small 0}}\right) _{%
{\small p}}^{{\small j}}})e^{q})
\end{equation*}%
Given that $p\succeq q,$ we have 
\begin{equation*}
f(x^{\left( {\small R}_{{\small pq}}^{{\small 0}}\right) _{{\small p}}^{%
{\small j}}}+(r-x_{p}^{\left( {\small R}_{{\small pq}}^{{\small 0}}\right) _{%
{\small p}}^{{\small j}}})e^{p}+(j-x_{q}^{\left( {\small R}_{{\small pq}}^{%
{\small 0}}\right) _{{\small p}}^{{\small j}}})e^{q})\geq f(x^{\left( 
{\small R}_{{\small pq}}^{{\small 0}}\right) _{{\small p}}^{{\small j}%
}}+(r-x_{p}^{\left( {\small R}_{{\small pq}}^{{\small 0}}\right) _{{\small p}%
}^{{\small j}}})e^{p}+(j-r)e^{p}),
\end{equation*}%
that is, 
\begin{equation*}
f(x^{\left( {\small R}_{{\small pq}}^{{\small 0}}\right) _{{\small p}}^{%
{\small j}}}+(r-x_{p}^{\left( {\small R}_{{\small pq}}^{{\small 0}}\right) _{%
{\small p}}^{{\small j}}})e^{p}+(j-r)e^{p})<z_{i}
\end{equation*}

which implies 
\begin{equation*}
f(x^{\left( {\small R}_{{\small pq}}^{{\small 0}}\right) _{{\small p}}^{%
{\small j}}}+(j-x_{p}^{\left( {\small R}_{{\small pq}}^{{\small 0}}\right) _{%
{\small p}}^{{\small j}}})e^{p})<z_{i}
\end{equation*}

because 
\begin{equation*}
x^{\left( {\small R}_{{\small pq}}^{{\small 0}}\right) _{{\small p}}^{%
{\small j}}}+(r-x_{p}^{\left( {\small R}_{{\small pq}}^{{\small 0}}\right) _{%
{\small p}}^{{\small j}}})e^{p}+(j-r)e^{p}=x^{\left( {\small R}_{{\small pq}%
}^{{\small 0}}\right) _{{\small p}}^{{\small j}}}+(j-x_{p}^{\left( {\small R}%
_{{\small pq}}^{{\small 0}}\right) _{{\small p}}^{{\small j}}})e^{p}\text{.}
\end{equation*}%
Therefore, we have 
\begin{equation*}
f(x^{\left( {\small R}_{{\small pq}}^{{\small 0}}\right) _{{\small p}}^{%
{\small j}}})\geq z_{i}\text{ and }f(x^{\left( {\small R}_{{\small pq}}^{%
{\small 0}}\right) _{{\small p}}^{{\small j}}}+(j-x_{p}^{\left( {\small R}_{%
{\small pq}}^{{\small 0}}\right) _{{\small p}}^{{\small j}}})e^{p})<z_{i}%
\text{.}
\end{equation*}

Hence 
\begin{equation*}
R_{pq}^{0}\in \mathcal{R}_{ip}^{+}\text{.}
\end{equation*}

$2)$ Suppose $\left\{ 
\begin{array}{l}
R\in \mathcal{R}_{iq}^{-} \\ 
x_{p}^{R}=r\geq s=x_{q}^{R}%
\end{array}%
\right. $ and let us show that $R_{pq}^{0}\in $ $\mathcal{R}_{ip}^{+}$.

Il suffices to show that 
\begin{equation*}
\left\{ 
\begin{array}{c}
f(x^{\left( {\small R}_{{\small pq}}^{{\small 0}}\right) _{{\small p}}^{%
{\small 1}}})<z_{i} \\ 
f(x^{\left( {\small R}_{{\small pq}}^{{\small 0}}\right) _{{\small p}}^{%
{\small 1}}}+(1-x_{p}^{\left( {\small R}_{{\small pq}}^{{\small 0}}\right) _{%
{\small p}}^{{\small 1}}})e^{p})\geq z_{i}\text{.}%
\end{array}%
\right.
\end{equation*}

Given that 
\begin{equation*}
x^{{\small R}_{{\small q}}^{{\small 1}}}=x^{\left( {\small R}_{{\small pq}}^{%
{\small 0}}\right) _{{\small p}}^{{\small 1}}}
\end{equation*}

we have 
\begin{equation*}
f(x^{\left( {\small R}_{{\small pq}}^{{\small 0}}\right) _{{\small p}}^{%
{\small 1}}})<z_{i}
\end{equation*}

because $f\left( x^{{\small R}_{{\small q}}^{{\small 1}}}\right) <z_{i}$.

Let us show that $f(x^{\left( {\small R}_{{\small pq}}^{{\small 0}}\right) _{%
{\small p}}^{{\small 1}}}+(1-x_{p}^{\left( {\small R}_{{\small pq}}^{{\small %
0}}\right) _{{\small p}}^{{\small 1}}})e^{p})\geq z_{i}$.

We have 
\begin{equation*}
f(x^{{\small R}_{{\small q}}^{{\small 1}}}+(s-x_{p}^{{\small R}_{{\small q}%
}^{{\small 1}}})e^{p})\geq z_{i}\text{ or }f(x^{{\small R}_{{\small q}}^{%
{\small 1}}}+(s-x_{p}^{{\small R}_{{\small q}}^{{\small 1}}})e^{p})<z_{i}
\end{equation*}

with $p$ and $q$ being assigned to the task $s$ in the task profile $x^{%
{\small R}_{{\small q}}^{{\small 1}}}+(s-x_{p}^{{\small R}_{{\small q}}^{%
{\small 1}}})e^{p}$.

\begin{itemize}
\item If $f(x^{{\small R}_{{\small q}}^{{\small 1}}}+(s-x_{p}^{{\small R}_{%
{\small q}}^{{\small 1}}})e^{p})\geq z_{i},$ then, given that
\end{itemize}

\begin{equation*}
x^{{\small R}_{{\small q}}^{{\small 1}}}+(s-x_{p}^{{\small R}_{{\small q}}^{%
{\small 1}}})e^{p}\leq x^{{\small R}_{{\small q}}^{{\small 1}}}+(1-x_{p}^{%
{\small R}_{{\small q}}^{{\small 1}}})e^{p}
\end{equation*}

by monotonicity, we have 
\begin{equation*}
f(x^{{\small R}_{{\small q}}^{{\small 1}}}+(1-x_{p}^{{\small R}_{{\small q}%
}^{{\small 1}}})e^{p})\geq f(x^{{\small R}_{{\small q}}^{{\small 1}%
}}+(s-x_{p}^{{\small R}_{{\small q}}^{{\small 1}}})e^{p})\geq z_{i},
\end{equation*}%
that is, 
\begin{equation*}
f(x^{{\small R}_{{\small q}}^{{\small 1}}}+(1-x_{p}^{{\small R}_{{\small q}%
}^{{\small 1}}})e^{p})\geq z_{i}\text{.}
\end{equation*}

Since 
\begin{equation*}
x^{{\small R}_{{\small q}}^{{\small 1}}}+(1-x_{p}^{{\small R}_{{\small q}}^{%
{\small 1}}})e^{p}=x^{\left( {\small R}_{{\small pq}}^{{\small 0}}\right) _{%
{\small p}}^{{\small 1}}}+(1-x_{p}^{\left( {\small R}_{{\small pq}}^{{\small %
0}}\right) _{{\small p}}^{{\small 1}}})e^{p}
\end{equation*}

we deduce that%
\begin{equation*}
f(x^{\left( {\small R}_{{\small pq}}^{{\small 0}}\right) _{{\small p}}^{%
{\small 1}}}+(1-x_{p}^{\left( {\small R}_{{\small pq}}^{{\small 0}}\right) _{%
{\small p}}^{{\small 1}}})e^{p})\geq z_{i}.
\end{equation*}

\begin{itemize}
\item If $f(x^{{\small R}_{{\small q}}^{{\small 1}}}+(s-x_{p}^{{\small R}_{%
{\small q}}^{{\small 1}}})e^{p})<z_{i}$, then , since
\end{itemize}

\begin{equation*}
x^{{\small R}_{{\small q}}^{{\small 1}}}+(1-x_{q}^{{\small R}_{{\small q}}^{%
{\small 1}}})e^{q}\leq x^{{\small R}_{{\small q}}^{{\small 1}}}+(1-x_{q}^{%
{\small R}_{{\small q}}^{{\small 1}}})e^{q}+(s-x_{p}^{{\small R}_{{\small q}%
}^{{\small 1}}})e^{p}=x^{{\small R}_{{\small q}}^{{\small 1}}}+(s-x_{p}^{%
{\small R}_{{\small q}}^{{\small 1}}})e^{p}+(1-x_{q}^{{\small R}_{{\small q}%
}^{{\small 1}}})e^{q}
\end{equation*}

by monotonicity, we have 
\begin{equation*}
f(x^{{\small R}_{{\small q}}^{{\small 1}}}+(s-x_{p}^{{\small R}_{{\small q}%
}^{{\small 1}}})e^{p}+(1-x_{q}^{{\small R}_{{\small q}}^{{\small 1}%
}})e^{q})\geq z_{i}\text{ for }f(x^{{\small R}_{{\small q}}^{{\small 1}%
}}+(1-x_{q}^{{\small R}_{{\small q}}^{{\small 1}}})e^{q})\geq z_{i}\text{.}
\end{equation*}%
\newline

We therefore have: 
\begin{equation*}
f(x^{{\small R}_{{\small q}}^{{\small 1}}}+(s-x_{p}^{{\small R}_{{\small q}%
}^{{\small 1}}})e^{p}+(1-x_{q}^{{\small R}_{{\small q}}^{{\small 1}%
}})e^{q})\geq z_{i}>f(x^{{\small R}_{{\small q}}^{{\small 1}}}+(s-x_{p}^{%
{\small R}_{{\small q}}^{{\small 1}}})e^{p});
\end{equation*}%
in the task profile $x^{{\small R}_{{\small q}}^{{\small 1}}}+(s-x_{p}^{%
{\small R}_{{\small q}}^{{\small 1}}})e^{p}$, with $p$ and $q$ are assigned
to the same task $s$.

Since $p\succeq q$ and 
\begin{equation*}
f(x^{{\small R}_{{\small q}}^{{\small 1}}}+(s-x_{p}^{{\small R}_{{\small q}%
}^{{\small 1}}})e^{p}+(1-x_{q}^{{\small R}_{{\small q}}^{{\small 1}%
}})e^{q})>f(x^{{\small R}_{{\small q}}^{{\small 1}}}+(s-x_{p}^{{\small R}_{%
{\small q}}^{{\small 1}}})e^{p})
\end{equation*}

we have 
\begin{equation*}
f(x^{{\small R}_{{\small q}}^{{\small 1}}}+(s-x_{p}^{{\small R}_{{\small q}%
}^{{\small 1}}})e^{p}+(1-s)e^{p})\geq f(x^{{\small R}_{{\small q}}^{{\small 1%
}}}+(s-x_{p}^{{\small R}_{{\small q}}^{{\small 1}}})e^{p}+(1-x_{q}^{{\small R%
}_{{\small q}}^{{\small 1}}})e^{q})\geq z_{i}\text{.}
\end{equation*}%
But 
\begin{equation*}
x^{{\small R}_{{\small q}}^{{\small 1}}}+(s-x_{p}^{{\small R}_{{\small q}}^{%
{\small 1}}})e^{p}+(1-s)e^{p}=x^{{\small R}_{{\small q}}^{{\small 1}%
}}+(1-x_{p}^{{\small R}_{{\small q}}^{{\small 1}}})e^{p}=x^{\left( {\small R}%
_{{\small pq}}^{{\small 0}}\right) _{{\small p}}^{{\small 1}%
}}+(1-x_{p}^{\left( {\small R}_{{\small pq}}^{{\small 0}}\right) _{{\small p}%
}^{{\small 1}}})e^{p}
\end{equation*}%
because 
\begin{equation*}
x^{{\small R}_{{\small q}}^{{\small 1}}}=x^{\left( {\small R}_{{\small pq}}^{%
{\small 0}}\right) _{{\small p}}^{{\small 1}}}\text{.}
\end{equation*}

That is, 
\begin{equation*}
f(x^{\left( {\small R}_{{\small pq}}^{{\small 0}}\right) _{{\small p}}^{%
{\small 1}}}+(1-x_{p}^{\left( {\small R}_{{\small pq}}^{{\small 0}}\right) _{%
{\small p}}^{{\small 1}}})e^{p})\geq z_{i}
\end{equation*}%
\newline
because 
\begin{equation*}
f(x^{{\small R}_{{\small q}}^{{\small 1}}}+(s-x_{p}^{{\small R}_{{\small q}%
}^{{\small 1}}})e^{p}+(1-s)e^{p})\geq z_{i}.
\end{equation*}

We therefore have 
\begin{equation*}
f(x^{\left( {\small R}_{{\small pq}}^{{\small 0}}\right) _{{\small p}}^{%
{\small 1}}})<z_{i}\text{ and }f(x^{\left( {\small R}_{{\small pq}}^{{\small %
0}}\right) _{{\small p}}^{{\small 1}}}+(1-x_{p}^{\left( {\small R}_{{\small %
pq}}^{{\small 0}}\right) _{{\small p}}^{{\small 1}}})e^{p})\geq z_{i}\text{.}
\end{equation*}

Hence 
\begin{equation*}
R_{pq}^{0}\in \mathcal{R}_{ip}^{-}\left( f\right) \text{.}
\end{equation*}

Suppose $\left\{ 
\begin{array}{l}
R\in \mathcal{R}_{iq}^{-},\text{ }x_{p}^{R}=s<r=x_{q}^{R}\text{ and } \\ 
R^{\prime }=\left(
sR,x^{R}+(x_{p}^{R}-x_{q}^{R})e^{q}+(x_{q}^{R}-x_{p}^{R})e^{p}\right) \in 
\mathcal{R}_{iq}^{-}%
\end{array}%
\right. $ and let us show that $R_{pq}^{0}\in $ $\mathcal{R}_{ip}^{-}$.

Since 
\begin{equation*}
x^{\left( {\small R}_{{\small pq}}^{{\small 0}}\right) _{{\small p}}^{%
{\small 1}}}=x^{{\small R}_{{\small q}}^{{\small 1}}}
\end{equation*}

it follows that 
\begin{equation*}
f(x^{\left( {\small R}_{{\small pq}}^{{\small 0}}\right) _{{\small p}}^{%
{\small 1}}})<z_{i}
\end{equation*}

as 
\begin{equation*}
f\left( x^{{\small R}_{{\small q}}^{{\small 1}}}\right) <z_{i}\text{.}
\end{equation*}

Let us show that $f(x^{\left( {\small R}_{{\small pq}}^{{\small 0}}\right) _{%
{\small p}}^{{\small 1}}}+(1-x_{p}^{\left( {\small R}_{{\small pq}}^{{\small %
0}}\right) _{{\small p}}^{{\small 1}}})e^{p}\geq z_{i}$.

We have by monotonicity 
\begin{equation*}
f(x^{\left( {\small R}_{{\small pq}}^{{\small 0}}\right) _{{\small p}}^{%
{\small 1}}}+(r-x_{p}^{\left( {\small R}_{{\small pq}}^{{\small 0}}\right) _{%
{\small p}}^{{\small 1}}})e^{p})<z_{i}\text{.}
\end{equation*}

Since 
\begin{equation*}
x^{{\small R}_{{\small q}}^{\prime {\small 1}}}+(1-x_{q}^{{\small R}_{%
{\small q}}^{\prime {\small 1}}})e^{q}=x^{\left( {\small R}_{{\small pq}}^{%
{\small 0}}\right) _{{\small p}}^{{\small 1}}}+(r-x_{p}^{\left( {\small R}_{%
{\small pq}}^{{\small 0}}\right) _{{\small p}}^{{\small 1}%
}})e^{p}+(1-x_{q}^{\left( {\small R}_{{\small pq}}^{{\small 0}}\right) _{%
{\small p}}^{{\small 1}}})e^{q}
\end{equation*}%
\newline
and 
\begin{equation*}
f(x^{{\small R}_{{\small q}}^{\prime {\small 1}}}+(1-x_{q}^{{\small R}_{%
{\small q}}^{\prime {\small 1}}})e^{q})\geq z_{i}
\end{equation*}

we have 
\begin{equation*}
f\left( x^{\left( {\small R}_{{\small pq}}^{{\small 0}}\right) _{{\small p}%
}^{{\small 1}}}+(r-x_{p}^{\left( {\small R}_{{\small pq}}^{{\small 0}%
}\right) _{{\small p}}^{{\small 1}}})e^{p}+(1-x_{q}^{\left( {\small R}_{%
{\small pq}}^{{\small 0}}\right) _{{\small p}}^{{\small 1}}})e^{q}\right)
\geq z_{i}
\end{equation*}%
that is, 
\begin{equation*}
f\left( x^{\left( {\small R}_{{\small pq}}^{{\small 0}}\right) _{{\small p}%
}^{{\small 1}}}+(r-x_{p}^{\left( {\small R}_{{\small pq}}^{{\small 0}%
}\right) _{{\small p}}^{{\small 1}}})e^{p}+(1-x_{q}^{\left( {\small R}_{%
{\small pq}}^{{\small 0}}\right) _{{\small p}}^{{\small 1}}})e^{q}\right)
>f(x^{\left( {\small R}_{{\small pq}}^{{\small 0}}\right) _{{\small p}}^{%
{\small 1}}}+(r-x_{p}^{\left( {\small R}_{{\small pq}}^{{\small 0}}\right) _{%
{\small p}}^{{\small 1}}})e^{p})\text{.}
\end{equation*}

From $p\succeq q,$ we have 
\begin{equation*}
f\left( x^{\left( {\small R}_{{\small pq}}^{{\small 0}}\right) _{{\small p}%
}^{{\small 1}}}+(r-x_{p}^{\left( {\small R}_{{\small pq}}^{{\small 0}%
}\right) _{{\small p}}^{{\small 1}}})e^{p}+(1-r)e^{p}\right) \geq f\left(
x^{\left( {\small R}_{{\small pq}}^{{\small 0}}\right) _{{\small p}}^{%
{\small 1}}}+(r-x_{p}^{\left( {\small R}_{{\small pq}}^{{\small 0}}\right) _{%
{\small p}}^{{\small 1}}})e^{p}+(1-x_{q}^{\left( {\small R}_{{\small pq}}^{%
{\small 0}}\right) _{{\small p}}^{{\small 1}}})e^{q}\right)
\end{equation*}%
that is, 
\begin{equation*}
f\left( x^{\left( {\small R}_{{\small pq}}^{{\small 0}}\right) _{{\small p}%
}^{{\small 1}}}+(r-x_{p}^{\left( {\small R}_{{\small pq}}^{{\small 0}%
}\right) _{{\small p}}^{{\small 1}}})e^{p}+(1-r)e^{p}\right) \geq z_{i}
\end{equation*}

which implies 
\begin{equation*}
f\left( x^{\left( {\small R}_{{\small pq}}^{{\small 0}}\right) _{{\small p}%
}^{{\small 1}}}+(1-x_{p}^{\left( {\small R}_{{\small pq}}^{{\small 0}%
}\right) _{{\small p}}^{{\small 1}}})e^{p}\right) \geq z_{i}
\end{equation*}

because \newline
\begin{equation*}
x^{\left( {\small R}_{{\small pq}}^{{\small 0}}\right) _{{\small p}}^{%
{\small 1}}}+(r-x_{p}^{\left( {\small R}_{{\small pq}}^{{\small 0}}\right) _{%
{\small p}}^{{\small 1}}})e^{p}+(1-r)e^{p}=x^{\left( {\small R}_{{\small pq}%
}^{{\small 0}}\right) _{{\small p}}^{{\small 1}}}+(1-x_{p}^{\left( {\small R}%
_{{\small pq}}^{{\small 0}}\right) _{{\small p}}^{{\small 1}}})e^{p}\text{.}
\end{equation*}

Therefore \newline
\begin{equation*}
f(x^{\left( {\small R}_{{\small pq}}^{{\small 0}}\right) _{{\small p}}^{%
{\small 1}}})<z_{i}\text{ and }f(x^{\left( {\small R}_{{\small pq}}^{{\small %
0}}\right) _{{\small p}}^{{\small 1}}}+(1-x_{p}^{\left( {\small R}_{{\small %
pq}}^{{\small 0}}\right) _{{\small p}}^{{\small 1}}})e^{p})\geq z_{i}\text{.}
\end{equation*}

Hence 
\begin{equation*}
R_{pq}^{0}\in \mathcal{R}_{ip}^{-}\text{.}
\end{equation*}
\end{proof}

\bigskip

For the sequel, for any $p,$ $q\in N$, we define the following sets:

\begin{equation*}
\begin{tabular}{l}
$\mathcal{D}_{ipq}^{+}=\left\{ R=(sR,x^{R})\in \mathcal{R}_{iq}^{+}\text{
such that }\left( \text{ 
\begin{tabular}{c}
$R^{\prime
}=(sR,x^{R}+(x_{p}^{R}-x_{q}^{R})e^{q}+(x_{q}^{R}-x_{p}^{R})e^{p})\in \text{ 
}\mathcal{R}_{iq}^{+}$ \\ 
and \\ 
$x_{p}^{R}>x_{q}^{R}$%
\end{tabular}%
}\right) \right\} $ \\ 
$\mathcal{D}_{ipq}^{-}=\left\{ R=(sR,x^{R})\in \mathcal{R}_{iq}^{-}\text{
such that }\left( \text{%
\begin{tabular}{l}
$R^{\prime
}=(sR,x^{R}+(x_{p}^{R}-x_{q}^{R})e^{q}+(x_{q}^{R}-x_{p}^{R})e^{p})\in 
\mathcal{R}$ \\ 
\multicolumn{1}{c}{and} \\ 
$x_{p}^{R}<x_{q}^{R}$%
\end{tabular}%
}\right) \right\} $%
\end{tabular}%
\end{equation*}

\subsubsection{Proof of Theorem 2}

\begin{proof}
Suppose that $p\succeq q$ and let us show that $\left\vert \mathcal{R}%
_{ip}^{+}\right\vert \geq \left\vert \mathcal{R}_{iq}^{+}\right\vert $ for
any $i\in \left\{ 1,...,k-1\right\} $.

Define the correspondence $\mathcal{\psi }_{pq}:$ $\mathcal{R}%
_{iq}^{+}\longrightarrow $ $\mathcal{R}_{ip}^{+}$ which to any $R\in $ $%
\mathcal{R}_{iq}^{+}$ associates an ordered allocation of positions as
follows:

\medskip

\begin{tabular}{|c|c|c|c|}
\hline
$R$ & $x_{p}^{R}<x_{q}^{R}$ & $x_{p}^{R}>x_{q}^{R}$ & $x_{p}^{R}=x_{q}^{R}$
\\ \hline
$\mathcal{\psi }_{pq}(R)$ & $\left\{ 
\begin{array}{l}
R_{pq}^{0}\text{ if }R\in \text{ }\mathcal{R}_{iq}^{+} \\ 
R_{pq}^{0}\text{ if }R\in \text{ }\mathcal{D}_{ipq}^{-} \\ 
R_{pq}\text{ if }R\notin \text{ }\mathcal{R}_{iq}^{+}\cup \text{ }\mathcal{D}%
_{ipq}^{-}%
\end{array}%
\right. $ & $\left\{ 
\begin{array}{l}
R_{pq}^{0}\text{ if }R\in \mathcal{R}_{iq}^{-} \\ 
R_{pq}^{0}\text{ if }R\in \text{ }\mathcal{D}_{ipq}^{+} \\ 
R_{pq}\text{ if }R\notin \mathcal{R}_{iq}^{-}\cup \text{ }\mathcal{D}%
_{ipq}^{+}%
\end{array}%
\right. $ & $R_{pq}$ \\ \hline
\end{tabular}

\medskip

The table above gives the value $\mathcal{\psi }_{pq}(R)$ of $R$ depending
on how $x_{p}^{R}$ and $x_{q}^{R}$ compare. We note that $\mathcal{\psi }%
_{pq}$ is well defined. This indeed follows from Lemmas 2-5. Let now show
that $\mathcal{\psi }_{pq}$\ is injective.

Remark that for any $R\in $ $\mathcal{R}_{iq}^{+}$, it is the case that $%
\mathcal{\psi }_{pq}(R)\in \{R_{pq}^{0},R_{pq}\}$.

Let $R$, $R^{\prime }\in $ $\mathcal{R}_{iq}^{+}$ such that $R\neq R^{\prime
}$. Let us show that $\mathcal{\psi }_{pq}\left( R\right) \neq \mathcal{\psi 
}_{pq}\left( R^{\prime }\right) $.

\begin{itemize}
\item If $sR\neq sR^{\prime }$, then $sR_{pq}^{0}\neq sR_{pq}^{\prime 0},$ $%
sR_{pq}^{0}\neq sR_{pq}^{\prime },$ $sR_{pq}\neq sR_{pq}^{\prime 0}$ and $%
sR_{pq}\neq sR_{pq}^{\prime }$. It follows that $R_{pq}^{0}\neq
R_{pq}^{\prime 0},$ $R_{pq}^{0}\neq R_{pq}^{\prime }$, $R_{pq}\neq
R_{pq}^{\prime 0},$ and $R_{pq}\neq R_{pq}^{\prime }$; hence $\mathcal{\psi }%
_{pq}\left( R\right) \neq \mathcal{\psi }_{pq}\left( R^{\prime }\right) $.

\item If $sR=sR^{\prime }$ and $x^{R}\neq x^{R^{\prime }}$, we consider the
following two cases:
\end{itemize}

\textbf{Case 1: }There exists a player $b\in N\backslash \left\{ p,q\right\} 
$ such that $x_{b}^{R}\neq x_{b}^{R^{\prime }}$.

Since $x_{b}^{{\small R}_{{\small pq}}^{{\small 0}}}$, $x_{b}^{{\small R}_{%
{\small pq}}}$, $x_{b}^{{\small R}_{{\small pq}}^{\prime {\small 0}}}$ and $%
x_{b}^{{\small R}_{{\small pq}}^{\prime }}$ are pairwise distinct, it is the
case that 
\begin{equation*}
x^{{\small R}_{{\small pq}}^{{\small 0}}}\neq x^{{\small R}_{{\small pq}%
}^{\prime {\small 0}}},x^{{\small R}_{{\small pq}}^{{\small 0}}}\neq x^{%
{\small R}_{{\small pq}}^{\prime }}x^{{\small R}_{{\small pq}}}\neq x^{%
{\small R}_{{\small pq}}^{\prime {\small 0}}}\text{ and }x^{{\small R}_{%
{\small pq}}}\neq x^{{\small R}_{{\small pq}}^{\prime }}
\end{equation*}

which implies 
\begin{equation*}
R_{pq}^{0}\neq R_{pq}^{\prime 0},R_{pq}^{0}\neq R_{pq}^{\prime 0},R_{pq}\neq
R_{pq}^{\prime 0},\text{ and }R_{pq}\neq R_{pq}^{\prime }\text{.}
\end{equation*}

Hence $\medskip $ 
\begin{equation*}
\mathcal{\psi }_{pq}\left( R\right) \neq \mathcal{\psi }_{pq}\left(
R^{\prime }\right) \text{.}
\end{equation*}

\textbf{Case 2: }For all player\textbf{\ }$b\in N\backslash \left\{
p,q\right\} $, $x_{b}^{R}=x_{b}^{R^{\prime }}$.

This case has three subcases $i)$, $ii)$ and $iii)$ below.

$i)$ If $x_{p}^{R}=x_{p}^{R^{\prime }}$ and $x_{q}^{R}\neq x_{q}^{R^{\prime
}}$, then $R_{pq}^{0}\neq R_{pq}^{\prime 0}$ and $R_{pq}\neq R_{pq}^{\prime
} $.

\begin{itemize}
\item If $x_{q}^{R^{\prime }}=x_{p}^{R}$, since
\end{itemize}

\begin{equation*}
x_{q}^{{\small R}_{{\small pq}}^{{\small 0}}}=x_{q}^{R}\text{, }x_{q}^{%
{\small R}_{{\small pq}}^{\prime }}=x_{p}^{R^{\prime
}}=x_{p}^{R}=x_{q}^{R^{\prime }}\neq x_{q}^{R}\text{, }x_{p}^{{\small R}_{%
{\small pq}}^{\prime {\small 0}}}=x_{p}^{R^{\prime }}\text{, and }x_{p}^{%
{\small R}_{{\small pq}}}=x_{q}^{R}\neq x_{q}^{R^{\prime
}}=x_{p}^{R}=x_{p}^{R^{\prime }}
\end{equation*}

it follows that 
\begin{equation*}
x_{q}^{{\small R}_{{\small pq}}^{{\small 0}}}\neq x_{q}^{{\small R}_{{\small %
pq}}^{\prime }}\text{ and }x_{p}^{{\small R}_{{\small pq}}^{\prime {\small 0}%
}}\neq x_{p}^{{\small R}_{{\small pq}}}
\end{equation*}

that is, 
\begin{equation*}
R_{pq}^{0}\neq R_{pq}^{\prime }\text{ and }R_{pq}^{\prime 0}\neq R_{pq}\text{%
.}
\end{equation*}

If $x_{q}^{R}=x_{p}^{R}$, since 
\begin{equation*}
x_{p}^{{\small R}_{{\small pq}}^{{\small 0}}}=x_{p}^{R}\text{, }x_{p}^{%
{\small R}_{{\small pq}}^{\prime }}=x_{q}^{R^{\prime }}\neq
x_{q}^{R}=x_{p}^{R}\text{, }x_{q}^{{\small R}_{{\small pq}}^{\prime {\small 0%
}}}=x_{q}^{R^{\prime }}\text{ and }x_{q}^{{\small R}_{{\small pq}%
}}=x_{p}^{R}=x_{q}^{R}\neq x_{q}^{R^{\prime }}
\end{equation*}

it follows that 
\begin{equation*}
x_{p}^{{\small R}_{{\small pq}}^{{\small 0}}}\neq x_{p}^{{\small R}_{{\small %
pq}}^{\prime }}\text{and }x_{q}^{{\small R}_{{\small pq}}^{\prime {\small 0}%
}}\neq x_{q}^{{\small R}_{{\small pq}}}
\end{equation*}

that is, 
\begin{equation*}
R_{pq}^{0}\neq R_{pq}^{\prime }\text{ and }R_{pq}^{\prime 0}\neq R_{pq}.
\end{equation*}

\begin{itemize}
\item If $x_{q}^{R^{\prime }}\neq x_{p}^{R}$ and $x_{q}^{R}\neq x_{p}^{R}$,
then, given that
\end{itemize}

\begin{equation*}
x_{p}^{{\small R}_{{\small pq}}^{{\small 0}}}=x_{p}^{R}\text{, }x_{p}^{%
{\small R}_{{\small pq}}^{\prime }}=x_{q}^{R^{\prime }}\neq x_{p}^{R}\text{, 
}x_{q}^{{\small R}_{{\small pq}}^{\prime {\small 0}}}=x_{q}^{R^{\prime }}%
\text{and }x_{q}^{{\small R}_{{\small pq}}}=x_{p}^{R}\neq x_{q}^{R^{\prime }}
\end{equation*}

we have 
\begin{equation*}
x_{p}^{{\small R}_{{\small pq}}^{{\small 0}}}\neq x_{p}^{{\small R}_{{\small %
pq}}^{\prime }}\text{and }x_{q}^{{\small R}_{{\small pq}}^{\prime {\small 0}%
}}\neq x_{q}^{{\small R}_{{\small pq}}}
\end{equation*}

that is, 
\begin{equation*}
R_{pq}^{0}\neq R_{pq}^{\prime }\text{ and }R_{pq}^{\prime 0}\neq R_{pq}\text{%
.}
\end{equation*}

$ii)$ If $x_{p}^{R}\neq x_{p}^{R^{\prime }}$ and $x_{q}^{R}=x_{q}^{R^{\prime
}}$, following a similar reasoning as in a), we can show that $%
R_{pq}^{0}\neq R_{pq}^{\prime 0}$, $R_{pq}\neq R_{pq}^{\prime }$, $%
R_{pq}^{0}\neq R_{pq}^{\prime }$ and $R_{pq}^{\prime 0}\neq R_{pq}$.\medskip

$iii)$ Suppose that $x_{p}^{R}\neq x_{p}^{R^{\prime }}$ and $x_{q}^{R}\neq
x_{q}^{R^{\prime }}$.

\begin{itemize}
\item If $x_{p}^{R}\neq x_{q}^{R^{\prime }}$, then:
\end{itemize}

$x_{p}^{{\small R}_{{\small pq}}^{{\small 0}}}=x_{p}^{R}\neq
x_{p}^{R^{\prime }}=x_{p}^{{\small R}_{{\small pq}}^{\prime {\small 0}}},$
that is, $R_{pq}^{0}\neq R_{pq}^{\prime 0};$

$x_{p}^{{\small R}_{{\small pq}}^{{\small 0}}}=x_{p}^{R}\neq
x_{q}^{R^{\prime }}=x_{p}^{{\small R}_{{\small pq}}^{\prime }},$ that is, $%
R_{pq}^{0}\neq R_{pq}^{\prime };$

$x_{q}^{{\small R}_{{\small pq}}^{\prime {\small 0}}}=x_{q}^{R^{\prime
}}\neq x_{p}^{R}=x_{q}^{{\small R}_{{\small pq}}},$ that is, $R_{pq}^{\prime
0}\neq R_{pq};$

$x_{p}^{{\small R}_{{\small pq}}}=x_{q}^{R}\neq x_{q}^{R^{\prime }}=x_{p}^{%
{\small R}_{{\small pq}}^{\prime }},$ that is, $R_{pq}^{\prime }\neq R_{pq}$.

\medskip

Similarly, if $x_{q}^{R}\neq x_{p}^{R^{\prime }}$, we get the same
conclusions.

\begin{itemize}
\item If $x_{p}^{R}=x_{q}^{R^{\prime }}$ and $x_{q}^{R}=x_{p}^{R^{\prime }}$%
, then $R^{\prime }=(sR,x^{R}+\left( x_{q}^{R}-x_{p}^{R}\right) e^{p}+\left(
x_{p}^{R}-x_{q}^{R}\right) e^{q})$, \newline
$R=(sR^{\prime },x^{R^{\prime }}+\left( x_{q}^{R^{\prime }}-x_{p}^{R^{\prime
}}\right) e^{p}+\left( x_{p}^{R^{\prime }}-x_{q}^{R^{\prime }}\right) e^{q})$
and $R$, $R^{\prime }\in $ $\mathcal{R}_{iq}^{+}$.
\end{itemize}

- Suppose that $R$\textbf{, }$R^{\prime }\in $ $R_{iq}$.

\begin{itemize}
\item If $sR\left( p\right) <sR\left( q\right) $ and $x_{p}^{R}<x_{q}^{R}$,
then $x_{q}^{R^{\prime }}<x_{p}^{R^{\prime }}$, and hence $\mathcal{\psi }%
_{pq}\left( R\right) =R_{pq}^{0}$ and $\mathcal{\psi }_{pq}\left( R^{\prime
}\right) =R_{pq}^{\prime 0}$, that is, $\mathcal{\psi }_{pq}\left( R\right)
\neq \mathcal{\psi }_{pq}\left( R^{\prime }\right) $ because $R_{pq}^{0}\neq
R_{pq}^{\prime 0}$.

\item If $sR\left( p\right) <sR\left( q\right) $ and $x_{p}^{R}>x_{q}^{R}$,
then $x_{p}^{R^{\prime }}<x_{q}^{R^{\prime }}$, and hence $\mathcal{\psi }%
_{pq}\left( R\right) =R_{pq}^{0}$ and $\mathcal{\psi }_{pq}\left( R^{\prime
}\right) =R_{pq}^{\prime 0},$ that is, $\mathcal{\psi }_{pq}\left( R\right)
\neq \mathcal{\psi }_{pq}\left( R^{\prime }\right) $ because $R_{pq}^{0}\neq
R_{pq}^{\prime 0}$.

\item If $sR\left( p\right) >sR\left( q\right) $, then $\mathcal{\psi }%
_{pq}\left( R^{\prime }\right) =R_{pq}^{\prime }$ and $\mathcal{\psi }%
_{pq}\left( R\right) =R_{pq}$, that is, $\mathcal{\psi }_{pq}\left( R\right)
\neq \mathcal{\psi }_{pq}\left( R^{\prime }\right) $ since $R_{pq}\neq
R_{pq}^{\prime }$.
\end{itemize}

- Suppose that $R$, $R^{\prime }\in \mathcal{R}_{iq}^{-}$.

\begin{itemize}
\item If $sR\left( p\right) <sR\left( q\right) $ and $x_{p}^{R}<x_{q}^{R}$,
then $x_{q}^{R^{\prime }}<x_{p}^{R^{\prime }}$, and hence $\mathcal{\psi }%
_{pq}\left( R\right) =R_{pq}^{0}$ and $\mathcal{\psi }_{pq}\left( R^{\prime
}\right) =R_{pq}^{\prime 0}$, that is, $\mathcal{\psi }_{pq}\left( R\right)
\neq \mathcal{\psi }_{pq}\left( R^{\prime }\right) $ because $R_{pq}^{0}\neq
R_{pq}^{\prime 0}$.

\item If $sR\left( p\right) <sR\left( q\right) $ and $x_{p}^{R}>x_{q}^{R}$,
then $x_{p}^{R^{\prime }}<x_{q}^{R^{\prime }}$, and hence $\mathcal{\psi }%
_{pq}\left( R\right) =R_{pq}^{0}$ and $\mathcal{\psi }_{pq}\left( R^{\prime
}\right) =R_{pq}^{\prime 0},$ which is equivalent to $\mathcal{\psi }%
_{pq}\left( R\right) \neq \mathcal{\psi }_{pq}\left( R^{\prime }\right) $
because $R_{pq}^{0}\neq R_{pq}^{\prime 0}$.

\item If $sR\left( p\right) >sR\left( q\right) $, then $\mathcal{\psi }%
_{pq}\left( R^{\prime }\right) =R_{pq}^{\prime }$ and $\mathcal{\psi }%
_{pq}\left( R\right) =R_{pq}$, that is, $\mathcal{\psi }_{pq}\left( R\right)
\neq \mathcal{\psi }_{pq}\left( R^{\prime }\right) $ because $R_{pq}\neq
R_{pq}$.
\end{itemize}

- Suppose that $R\in $ $\mathcal{R}_{iq}^{+},$ $R^{\prime }\in \mathcal{R}%
_{iq}^{-}$. Then, since $p\succeq q$, we have $x_{p}^{R}<x_{q}^{R}$ and $%
x_{p}^{R^{\prime }}>x_{q}^{R^{\prime }}$.

\begin{itemize}
\item If $sR\left( p\right) <sR\left( q\right) $, then $\mathcal{\psi }%
_{pq}\left( R\right) =R_{pq}^{0}$ and $\mathcal{\psi }_{pq}\left( R^{\prime
}\right) =R_{pq}^{\prime 0}$, that is, $\mathcal{\psi }_{pq}\left( R\right)
\neq \mathcal{\psi }_{pq}\left( R^{\prime }\right) $ because $R_{pq}^{0}\neq
R_{pq}^{\prime 0}$.

\item If $sR\left( p\right) >sR\left( q\right) $, then $\mathcal{\psi }%
_{pq}\left( R\right) =R_{pq}$ and $\mathcal{\psi }_{pq}\left( R^{\prime
}\right) =R_{pq}^{\prime },$ that is, $\mathcal{\psi }_{pq}\left( R\right)
\neq \mathcal{\psi }_{pq}\left( R^{\prime }\right) $ because $R_{pq}\neq
R_{pq}^{\prime }.$
\end{itemize}

- Suppose that $R\in \mathcal{R}_{iq}^{-},$ $R^{\prime }\in $ $\mathcal{R}%
_{iq}^{+}$. Then, since $p\succeq q$, we have $x_{p}^{R}>x_{q}^{R}$ et $%
x_{p}^{R^{\prime }}<x_{q}^{R^{\prime }}$.

\begin{itemize}
\item If $sR\left( p\right) <sR\left( q\right) ,$ then $\mathcal{\psi }%
_{pq}\left( R\right) =R_{pq}^{0}$ and $\mathcal{\psi }_{pq}\left( R^{\prime
}\right) =R_{pq}^{\prime 0}$, that is, $\mathcal{\psi }_{pq}\left( R\right)
\neq \mathcal{\psi }_{pq}\left( R^{\prime }\right) $ because $R_{pq}^{0}\neq
R_{pq}^{\prime 0}$.

\item If $sR\left( p\right) >sR\left( q\right) $, then $\mathcal{\psi }%
_{pq}\left( R\right) =R_{pq}$ and $\mathcal{\psi }_{pq}\left( R^{\prime
}\right) =R_{pq}^{\prime },$ that is, $\mathcal{\psi }_{pq}\left( R\right)
\neq \mathcal{\psi }_{pq}\left( R^{\prime }\right) $ because $R_{pq}\neq
R_{pq}^{\prime }$.
\end{itemize}

We conclude that for any $R$, $R^{\prime }\in $ $\mathcal{R}_{iq}^{+}$ such
that $R\neq R^{\prime }$, $\mathcal{\psi }_{pq}\left( R\right) \neq \mathcal{%
\psi }_{pq}\left( R^{\prime }\right) $, that is, $\mathcal{\psi }_{pq}$ is
injective, which implies that $\left\vert \mathcal{R}_{ip}^{+}\right\vert
\geq $ $\left\vert \mathcal{R}_{iq}^{+}\right\vert $.
\end{proof}

\bigskip\ \pagebreak

\end{document}